\documentclass{article}

\usepackage[utf8]{inputenc}
\usepackage[T1]{fontenc}
\usepackage{amsmath, amsthm, amssymb}
\usepackage{scalerel}
\usepackage{authblk}
\usepackage[square,sort,comma,numbers]{natbib}

\usepackage{mathtools}
\usepackage{geometry}

\usepackage{wasysym}
\usepackage{fancyhdr}
\usepackage{fontawesome}
\usepackage{isotope}
\usepackage{verbatim}
\usepackage{graphicx}

\makeatletter
\let\origps@plain\ps@plain
\newcommand\MakePlainPagestyleEmpty{\let\ps@plain\ps@empty}
\newcommand\MakePlainPagestylePlain{\let\ps@plain\origps@plain}
\makeatother
\usepackage{blindtext}

\newtheorem{dfn}{Definition}[section]
\newtheorem{prop}{Proposition}[section]
\newtheorem{thm}{Theorem}[section]
\newtheorem{cor}{Corollary}[section]
\newtheorem{lem}{Lemma}[section]
\numberwithin{equation}{section}

\def\msquare{\mathord{\scalerel*{\Box}{gX}}}

\newcommand{\bca}{\begin{cases}}
\newcommand{\eca}{\end{cases}}

\newcommand{\lb}{\left(}

\newcommand{\rb}{\right)}
\newcommand{\lmb}{\left[}
\newcommand{\rmb}{\right]}

\newcommand{\lma}{\left\{}
\newcommand{\rma}{\right\}}

\newcommand{\csubset}{\subset\subset}
\newcommand{\co}{\Omega}
\newcommand{\po}{{\partial \Omega}}
\newcommand{\R}{\mathbb{R}}

\newcommand{\x}{\mathbf{x}}

\newcommand{\z}{\mathbf{z}}

\newcommand{\gd}{\nabla}
\newcommand{\rta}{\rightarrow}
\newcommand{\lw}{\left|}
\newcommand{\rw}{\right|}
\newcommand{\be}{\begin{equation}}
\newcommand{\ee}{\end{equation}}
\newcommand{\N}{\textbf{N}}
\newcommand{\bt}{\begin{thm}}
\newcommand{\et}{\end{thm}}
\newcommand{\bc}{\begin{cor}}
\newcommand{\ec}{\end{cor}}
\newcommand{\bl}{\begin{lem}}
\newcommand{\el}{\end{lem}}
\newcommand{\norm}[1]{\left\lVert#1\right\rVert}
\newcommand{\normm}[1]{{\left\vert\kern-0.25ex\left\vert\kern-0.25ex\left\vert #1 
    \right\vert\kern-0.25ex\right\vert\kern-0.25ex\right\vert}}

\newcommand\bigfrown[2][\textstyle]{\ensuremath{%
  \array[b]{c}\text{\resizebox{6ex}{.7ex}{$#1\frown$}}\\[-1.3ex]#1#2\endarray}}

\newcommand{\weakcv}{\rightharpoonup}

\usepackage{natbib}
\usepackage{graphicx}
\usepackage{caption}
\usepackage{subcaption}

\newcommand\restr[2]{{
  \left.\kern-\nulldelimiterspace 
  #1 
  \vphantom{\big|} 
  \right|_{#2} 
  }}

\title{\textsc{Homogenization of Enhancing Thin Layers}}
\author{Zhonggan Huang\\  \texttt{huangzhonggan2020@outlook.com}}
\affil{Department of Mathematics, Southern University of Science and Technology}

\date{October 16, 2020}

\begin{document}

\maketitle

\setcounter{page}{1}
\begin{abstract}
   This paper derives an explicit formula for the effective diffusion tensor by using the solutions to some effective cell problems after homogenizing Road effective boundary conditions (EBCs). The concept of Road EBCs was proposed recently by H. Li and X. Wang, and in this paper, we extend the effective conditions on closed curves to those on patterns, especially on the included nodes. We also prove that homogenization process commutes with the derivation of Road EBCs. By analyzing the effective diffusion tensor, we obtain several rules for maximizing its trace with given Road-effective-diffusivity/scale and length/scale in each cell and define a notion of balanced patterns. Moreover, we give an estimate of the trace of the effective diffusion tensor of patterns satisfying some connectedness as Road-effective-diffusivity/scale goes to infinity.
\end{abstract}


\section{Introduction}

The major issue is the commutation between the classical periodic homogenization and the ``Road" effective boundary condition (EBC) proposed by Li and Wang \citep{Li-2017,Li2020}. For \(\epsilon,\delta>0\), we consider the following family of problems,
\be
\begin{cases}
-\gd\cdot(\bar{\sigma}(\epsilon^{-1}\x)\gd u^{\delta,\epsilon}(\x))=f(\x),&\x\in\co\csubset\R^2,\\
u^{\delta,\epsilon}(\x)=0,&\x\in\po,
\end{cases}\label{origin}
\ee
where \(\bar{\sigma}(\x)\) is 1-periodic, and (in \(\msquare=(0,1)^2\)) equals \(\sigma=O(1/\delta)\) when \(\x\) lies in a \(\delta\)-net \(R_\delta^+\) defined by using a collection \(\mathcal{G}\), called ``pattern", of \(C^2\) curves in the flat torus \(\mathbb{T}^2\), and 1 elsewhere. Without losing generality, we assume that the curves in \(\mathcal{G}\) do not intersect one another at interior points and any two curves cannot compose a new smooth curve. 

We invoke a result in \citep{trudinger} that given \(\gamma\in\mathcal{G}\), when \(\delta>0\) is small, there will be a \(C^1\) curvilinear coordinate system \((s,\tau)\), where \(s\in[0,l]\) along the tangents, and \(\tau\in(-\delta,\delta)\) along some smooth unit normal field are both unit speed. Define \(R_\delta^\gamma\) to be the image of \((0,l)\times(-\delta,\delta)\), and \(R_{\delta,+}^\gamma=R_\delta^\gamma\cap\{\tau>0\}\). We further define \[\Gamma_1=\bigcup_{\gamma\in\mathcal{G}} \bigcap_{\delta>0} R_\delta^\gamma, R_\delta=\bigcup_{\gamma\in\mathcal{G}}R_\delta^\gamma\text{ and } R_{\delta}^+=\bigcup_{\gamma\in\mathcal{G}}R_{\delta,+}^\gamma.\]
After choosing a proper periodic extension of \(\mathcal{G}\) to \(\R^2\), we may consider the above curves and sets as distributed periodically in \(\R^2\). Notice that given a periodic extension of \(\mathcal{G}\), its planar translation is still an appropriate extension. In the major part of this article, we will ignore this small difference caused by translation because it will not change the ultimate effective model, but in Section 4, we will point out its importance in further analysis. A typical example is illustrated in Figure \ref{1typicalexample}.

\begin{figure}[htbp]
     \begin{subfigure}{0.3\textwidth}
         \includegraphics[width=1.8\textwidth]{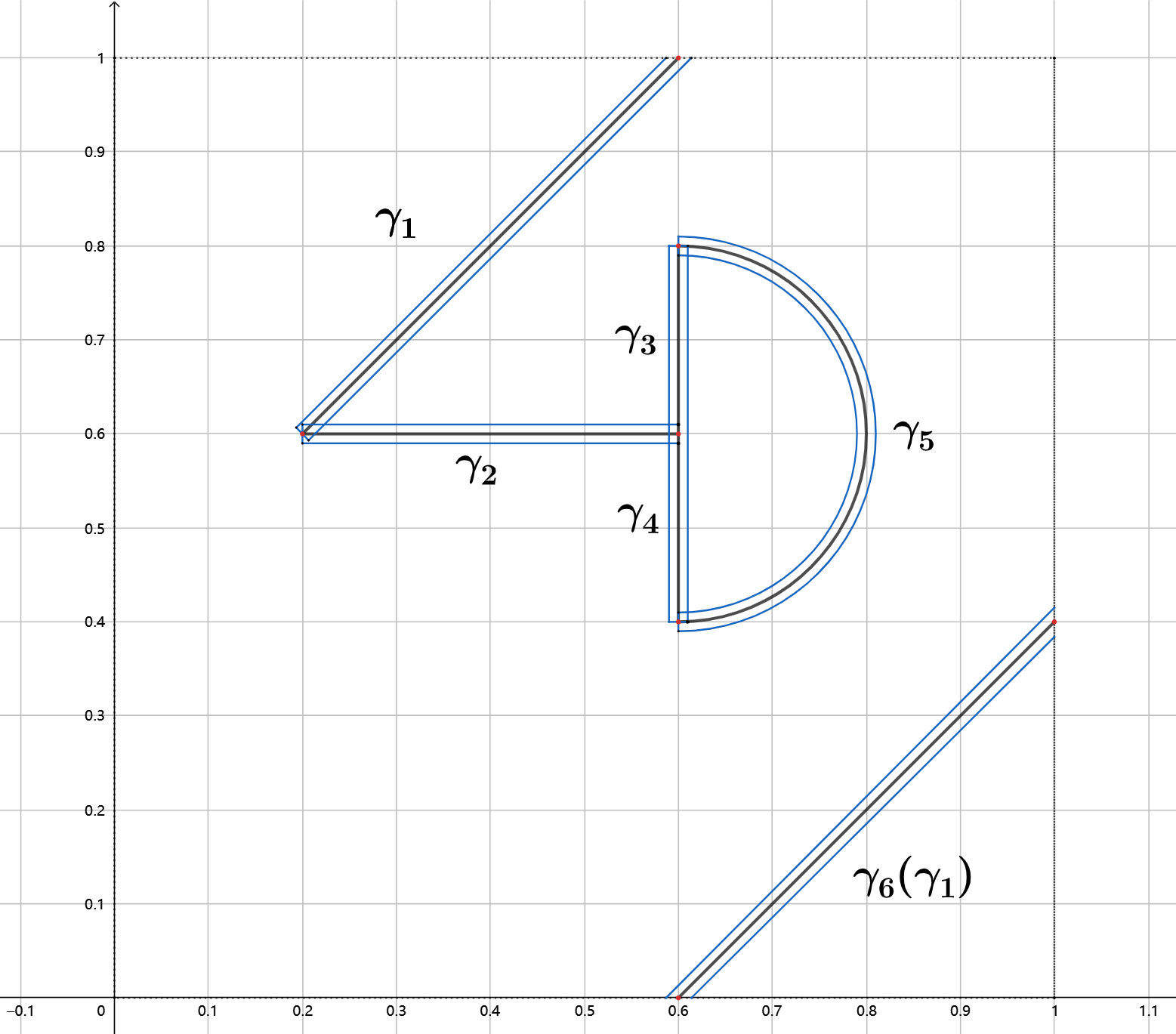}
         \caption{Viewed in one cell.}
         \label{a}
     \end{subfigure}
     \hspace{1.2in}
     \begin{subfigure}{0.3\textwidth}
         \includegraphics[width=1.6\textwidth]{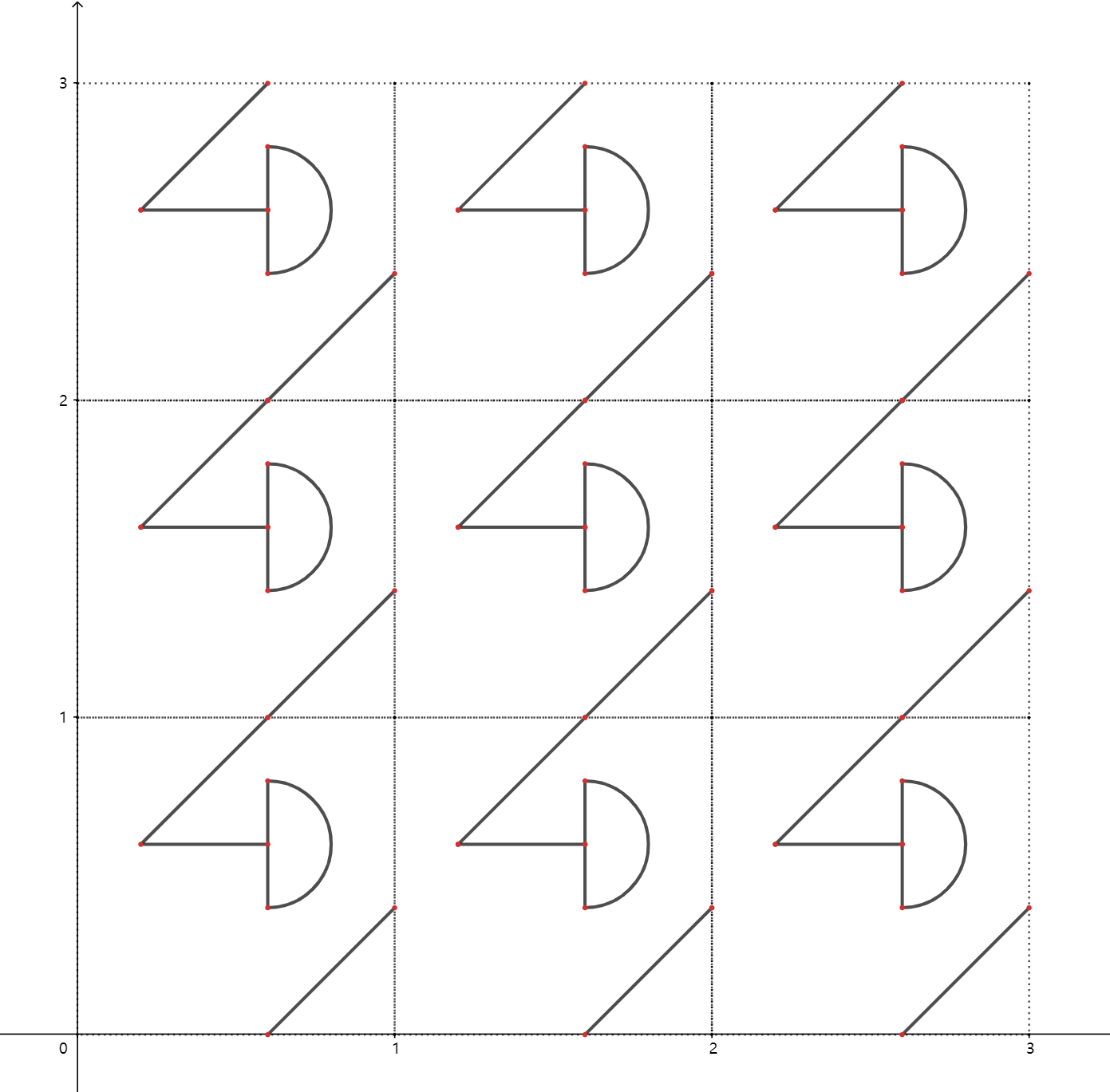}
         \caption{Viewed in a nine-palace.}
         \label{b}
     \end{subfigure}
     
        \caption{This is a qualitative illustration of a typical pattern extended periodically in \(\R^2\); the blue curves enclose \(R_\delta^{\gamma_i}\) for \(1\le i\le 5\) (``\(\gamma_6\)" and \(\gamma_1\) are regarded as the one single curve); even though \(\gamma_3\) and \(\gamma_4\) form a straight line, the intersection of which and \(\gamma_2\) forces us to consider the two arcs separately; each \(\gamma_i\) splits \(R_\delta^{\gamma_i}\) into two (curved) quadrangles, and \(R_{\delta,+}^\gamma\) can be chosen to be any of the two.}
        \label{1typicalexample}
\end{figure}

Physically, these \(\delta\)-roads \(R_\delta^+\) represent thin wires that possess a higher diffusion rate relative to the materials elsewhere, and the solution \(u^{\delta,\epsilon}\) to the full model (\ref{origin}) could describe the effects of the roads in a sufficient way. However, there are many obstacles in directly considering the full equation either analytically or computationally. In PDE theory, the solutions don't have much regularity because of the discontinuity of diffusion rates, and in computational PDE theory, obtaining the solution requires huge computational resources due to the small scales involved in the equation.

To avoid these issues, we exploit the periodic homogenization and EBC theories, and discuss the effective model after sending \(\delta,\epsilon\rta0\). One may immediately observe two ways that possibly lead to two models. The first model is the relatively traditional one: sending \(\epsilon\rta0\) and then \(\delta\rta0\). By periodic homogenization, the first step gives us a homogenized model with constant diffusion tensor \(\Sigma_\delta\) dependent on \(\delta\), and so this method mainly focuses on the limit of \(\Sigma_\delta\) as \(\delta\rta0\).

The second way is to reverse the above process: sending \(\delta\rta0\) first and then \(\epsilon\rta0\). According to Li and Wang's theory \citep{Li-2017,Li2020}, when sending \(\delta\rta0\), we know that at scale \(\epsilon>0\), the road width should be \(\epsilon\delta\), and if \(\sigma\delta\rta a>0\), then there will be a Road EBC on each copy of \(\epsilon\gamma\) for each \(\gamma\in\mathcal{G}\),
\[
a\epsilon u_{ss}^{0,\epsilon}=\frac{\partial (u^{0,\epsilon})^{-}}{\partial \vec{n}}-\frac{\partial (u^{0,\epsilon})^{+}}{\partial \vec{n}},
\]
in which the notations will be discussed in the ensuing theorem (``\(a\epsilon\)" represents the \emph{Road effective diffusivity}, ``\(\delta\epsilon\)" the \emph{road width} and ``\(\epsilon\)" the \emph{scale}). At this stage, sending \(\epsilon\rta0\) means to do homogenization of interior boundary conditions. It is worth mentioning that in this article, we find the EBCs on a collection of \(C^2\) curves including nodes instead of disjoint \(C^2\) closed curves as presented in \citep{Li2020}. 
Certainly, the arising conditions on the nodes in this case will be carefully discussed in section 2.4. The following theorems  are our main results.
\begin{thm}
The two intermediate effective models are actually the same. More precisely, on the one hand, by sending \(\epsilon\rta0\) first, we obtain a solution \(u^{\delta,0}\) to the classical homogenized problem with index \(\delta\). On the other hand, by sending \(\delta\rta0\) first, we obtain \(u^{0,\epsilon}\), which is isotropic but satisfies Road pattern conditions with road diffusivity \(a\epsilon\). If one sends \(\delta,\epsilon\rta0\) (for some subsequence) respectively, then both two sequences converge weakly to \(u^{0,0}\) in \(H_0^1(\co)\) as \(\epsilon,\delta\rta0\). Moreover, \(u^{0,0}\) satisfies the effective equation
\be
\begin{cases}
-\gd\cdot(\Sigma_0\gd u^{0,0})(\x)=f(\x),&\x\in\co,\\
u^{0,0}(\x)=0,&\x\in\po,
\end{cases}
\ee
where for \(l,k=1,2\),
\be
\lb\Sigma_0\rb_{kl}=\delta_{kl}+a\int_{\Gamma_1} (\bar{w}_k)_s(x_l)_s ds,
\ee
and \(v=\bar{w}_k-x_k\) satisfies the following conditions in the weak sense
\be
\bca
-\Delta v(\x)=0,&\x\in\msquare\backslash\Gamma_1,\\
v(\x)+x_k \text{ satisfies Road Pattern Conditions with road diffusivity }a,&\x\in\Gamma_1,\\
v\text{ is 1-periodic in }\R^2,\text{ and }\int_{\msquare} v=0,&
\eca\label{effectivecorrector1}
\ee
where a proper function \(u\) satisfies \emph{Road Pattern Condition} with road diffusivity \(a\) if
\begin{itemize}
    \item[1.] For each arc \(\gamma\in\mathcal{G}\), if \(s\) represents a unit-speed reparametrization and \(\vec{n}\) a unit normal field on \(\gamma\), then the domain is locally split into two by the arc, and we may call the one directed by \(\vec{n}\) positive, and another negative. On this arc, \(u\) satisfies \(u^+=u^-\) and \(a u_{ss}= \frac{\partial u^-}{\partial\vec{n}}-\frac{\partial u^+}{\partial\vec{n}}\) (if \(u=v+x_k\), we have \(a v_{ss}+a (x_k)_{ss}= \frac{\partial v^-}{\partial\vec{n}}-\frac{\partial v^+}{\partial\vec{n}}\), while if \(u=u^{0,\epsilon}\), the Road effective diffusivity \(a\) should be replaced by \(a\epsilon\));
    \item[2.] \(\restr{u}{\gamma}(V)=0\) for all \(\gamma\) that intersects the outer boundary at \(V\in\po\) (this condition corresponds to the case that \(u=u^{0,\epsilon}\) and should be removed for the case \(u=v+x_k\));
    \item[3.] If \(V\) is contained in \(\co\), and \(\gamma,\gamma'\in\mathcal{G}\) are two arcs that join at \(V\), then \(\restr{u}{\gamma}(V)=\restr{u}{\gamma'}(V)\) (if \(u=v+x_k\), we have \(\restr{v}{\gamma}(V)=\restr{v}{\gamma'}(V)\));
    \item[4.] For \(V\in\co\), if \(\mathcal{G}_V\) is the collection of arcs that join at it, we have
    \[
    \sum_{\gamma\in\mathcal{G}_V} u_{s_\gamma}(V)=0,
    \]
    where \(s_\gamma\) is the arc-length parametrization of \(\gamma\) starting from \(V\) (if \(u=v+x_k\), then we have \(\sum_{\gamma\in\mathcal{G}_V} (v+x_k)_{s_\gamma}(V)=0\)).
\end{itemize}\label{1.1}
\end{thm}

To evaluate the efficiency of a pattern \(\mathcal{G}\), we consider the trace of the effective diffusion tensor \(\Sigma_0\), and obtain the following result.
\begin{thm}
Among all regular (see Definition (\ref{regularpattern})) 1-periodic patterns \(\mathcal{G}\) with total length \(l>0\) of \(\Gamma_1\) in \(\msquare\), we have \(tr(\Sigma_0)\le 2+a l\), and equality holds if and only if all arcs of \(\mathcal{G}\) satisfy the following conditions
\begin{itemize}
    \item[1.] Every arc is straight;
    \item[2.] There are at least two different arcs joining at each node;
    \item[3.] \textbf{(Balance Condition)} Let \(V\) be a node and \(\gamma_1,\cdots,\gamma_m\in\mathcal{G}\) be distinct linear segments that join at \(V\). Denoting by \(V_i\) another end point of \(\gamma_i,\,i=1,\cdots,m\), we have
    \be
    \sum_{i=1}^m \frac{V-V_i}{|V-V_i|}=\mathbf{0}.
    \ee
\end{itemize}
\label{1.2}
\end{thm}
\noindent We call a 1-periodic pattern \textbf{balanced} if it satisfies the above conditions, and it is worth mentioning that this definition is beyond the periodicity restriction. The above rules give clues to find optimal patterns regarding the total effects of the enhancing thin layers. Since finding all such patterns will be a distinct subject, in this article, we will only exhibit several cases that best illustrate these phenomena.

This article is composed mainly of three parts. In section 2, we give several simple inequalities, a density result, and derivation and description of EBCs on patterns. In section 3, we present the proof of Theorem (\ref{1.1}) from two perspectives. In section 4, we first prove Theorem (\ref{1.2}) and list several balanced patterns, and then we give an estimate of \(tr(\Sigma_0)\) as \(a\rta\infty\).

\section{Preliminary}
\subsection{Several Useful Facts}

\noindent Suppose \(u\) is a \(C^1\) function on some subdomain of \(\R^2\), then we have the following facts. Moreover, by density arguments, these facts can be applied to \(H^1\) functions.

\begin{lem}
Let \(\xi\) be the linear segment \((0,l)\times\{0\}\subset\R^2\), with \(l>0\), and \(\z=(z_1,z_2)\in\R^2\) a vector such that \(z_2>0\). Denoting by \(A\) the parallelogram spanned by \(\z\) and \((l,0)\), then we have
\begin{itemize}
    \item 
    \[
    \lw\int_{\xi+\z} u dx -\int_{\xi} u dx\rw\le \frac{|\z|}{z_2} \int_{A} |\gd u| dxdy ;
    \]
    \item \[\lw\int_{\xi} u dx \rw \le \frac{1}{z_2} \int_{A} |u| dxdy  + \frac{|\z|}{z_2} \int_{A} |\gd u| dxdy .\]
\end{itemize}\label{translationerror}
\end{lem}
\begin{proof}
For \(0<x<l\), we have by Fundamental Theorem of Calculus
\[
\begin{split}
    u(x+z_1,z_2)-u(x,0)&= \int_0^1 \frac{d}{ds}\lmb u(x+sz_1,sz_2)\rmb ds\\
                      &= \int_0^1 (\gd u\cdot \z)(x+sz_1,sz_2) ds.
\end{split}
\]
Thus, we have
\[
\begin{split}
   \int_0^l \lw u(x+z_1,z_2)-u(x,0) \rw dx &\le |\z|  \int_0^l \int_0^1 |\gd u| (x+sz_1,sz_2) ds dx\\
                                           &= \frac{|\z|}{z_2}\int_0^{z_2}\int_{y\frac{z_1}{z_2}}^{l+y\frac{z_1}{z_2}} |\gd u|(x,y) dxdy\\
                                                         &=\frac{|\z|}{z_2} \int_{A} |\gd u| dxdy.
\end{split}
\]
As for the second inequality, we set \(\z\) to be \(s\z\) for some \(s\in(0,1]\), and thus
\[
   \lw\int_{\xi} u dx \rw\le  \int_{\xi+s\z} |u| dx + \frac{|s\z|}{sz_2}\int_{A} |\gd u| dxdy.
\]
Integrating both sides w.r.t. \(s\) from \(0\) to \(1\), we obtain the result.
\end{proof}

\begin{lem}
Let \(u(r,\theta)\coloneqq u(r \cos{\theta},r\sin{\theta})\). Then for fixed \(\theta_0\in(0,\pi]\) and \(l>0\), we denote by \(B\) the sector \(\{0<r<l,0<\theta<\theta_0\}\), and we have
\[
\lw\int_0^l u(r,\theta_0) dr-\int_0^l u(r,0) dr\rw\le \int_B |\gd u| dxdy.
\]\label{rotationfact}
\end{lem}
\begin{proof}
By FTC, we know that
\[
u(r,\theta_0)-u(r,0)=\int_0^{\theta_0} u_\theta(r,\theta) d\theta,
\]
and then
\[
\begin{split}
    \int_0^l |u(r,\theta_0)-u(r,0)|dr&\le \int_0^l \int_0^{\theta_0} \lw u_\theta\rw(r,\theta) d\theta dr\\
                                                    &=  \int_0^l \int_0^{\theta_0} \frac{\lw u_\theta\rw(r,\theta)}{r}r  d\theta dr\\
                                                    &\le \int_B |\gd u| dx dy.
\end{split}
\]
\end{proof}

\begin{lem}
Let \(\xi=(0,l)\times\{0\}\) be defined in prior lemmas, and \(g(x)\) some continuous function on \(\xi\) (here \(\xi\) is considered as part of \(x\)-axis, and the image of \(g\) is the \(y\)-axis.) We then have the following estimate
\[
\int_0^l \lw u(x,g(x))-u(x,0)\rw dx \le \int_{D(g)} |\gd u|(x,y) dxdy,
\]
where \(D(g)\) is the domain enclosed by \(y=0\), \(y=g(x)\), \(x=0\) and \(x=l\).
\label{functiontranslate}
\end{lem}

\begin{lem}
Let \(0<\lambda<1,\,l>0,a>0\), then we have
\[
\lw\frac{1}{\lambda}\int_0^{\lambda l} u(x,a) dx -\int_0^{l} u(x,0) dx\rw\le \frac{\sqrt{2+l^2/a^2}}{\lambda^2} \int_{C} |\gd u| dxdy,
\]
where \(C\) is the (closed) trapezoid induced by the four vertices: \((0,0),(l,0),(0,a),(\lambda l,a)\).\label{rescalefact}
\end{lem}
\begin{proof}
We define
\[
\begin{split}
    G:[0,l]\times[0,a] &\longrightarrow C,\\
       (\alpha,\beta)&\longmapsto \lb\lb1-\frac{1-\lambda}{a}\beta\rb\alpha,\beta\rb.
\end{split}
\]
Then, it can be shown that \(G\) is a nondegenerate diffeomorphism. We also have
\begin{itemize}
    \item \( u(G(\alpha,0))=u(\alpha,0)\), and \(u(G(\alpha,a))=u(\lambda \alpha,a)\);
    \item Let \(\gd_G\) be derivatives w.r.t. \(\alpha\) and \(\beta\), we have
    \[
    |\gd u|^2= \gd_G u \cdot \lb H \gd_G u\rb,
    \]
    where
    \[
    H=DGDG^T=\lb\begin{matrix}
    \lb1-\frac{1-\lambda}{a}\beta\rb^2+\lb\frac{1-\lambda}{a}\alpha\rb^2 & -\frac{1-\lambda}{a}\alpha\\
    -\frac{1-\lambda}{a}\alpha&1
    \end{matrix}\rb.
    \]
    
\end{itemize}
Observe that
    \[
    \begin{split}
        & 1\ge \det(H)=\lb1-\frac{1-\lambda}{a}\beta\rb^2\ge\lambda^2>0\\
        & 2+\frac{l^2}{a^2}(1-\lambda)^2\ge Tr(H)=1+\lb1-\frac{1-\lambda}{a}\beta\rb^2+\lb\frac{1-\lambda}{a}\alpha\rb^2\ge 1+\lambda^2>0.
    \end{split}
    \]
    Thus, we have that the two eigenvalues \(\lambda_1,\lambda_2\) satisfy
    \[
    \frac{\lambda^2}{2+\frac{l^2}{a^2}(1-\lambda)^2}\le \lambda_1,\lambda_2\le  \frac{2+\frac{l^2}{a^2}(1-\lambda)^2}{\lambda^2}.
    \]
    Now, by Lemma (\ref{translationerror}), we have
    \[
    \begin{split}
      \lw\int_0^l u(G(\alpha,a)) d\alpha  - \int_0^l u(G(\alpha,0)) d\alpha \rw&\le  \int_0^l\int_0^a |\gd_G u| d\beta d\alpha \\
      &=  \int_0^l\int_0^a \frac{|\gd_G u|}{\sqrt{\det(H)}} \sqrt{\det(H)}  d\beta d\alpha\\
      &\le  \int_C \frac{\sqrt{2+l^2/a^2}}{\lambda^2} |\gd u| dxdy .
    \end{split}
     \]

\end{proof}

\subsection{A Density Result}

Within some domain \(\co\subset\R^2\), we consider a finite collection \(\mathcal{G}\), called a \emph{pattern}, of \(C^2\) curves (also regular at end points), each two of which intersect at most at their end points. The union of such curves (including end points) will be denoted by \(\Gamma\subset\bar{\co}\). On each element \(\gamma\in\mathcal{G}\), we may consider a Sobolev space \(H^1(\gamma)\) defined by naturally taking a unit speed reparametrization. Meanwhile, any \(u\in H^1(\co)\) by Trace Theorem has an \(L^2(\gamma)\) trace on each \(\gamma\in\mathcal{G}\).
\begin{dfn}
We define
\be
H_\Gamma^{1,0}(\co)\coloneqq\lma u\in H_0^1(\co);\restr{u}{\gamma}\in H^1(\gamma) \text{ for all }\gamma\in\mathcal{G}\rma.
\ee
This function space endowed with the inner product 
\[
\begin{split}
    (u,v)_{H_\Gamma^{1,0}(\co)}&=\int_{\co}\gd u\cdot\gd v + \sum_{\gamma\in\mathcal{G}}\int_{\gamma} u_s v_s\\
                            &\eqqcolon \int_{\co}\gd u\cdot\gd v + \int_{\Gamma} u_s v_s
\end{split}
\]
becomes a Hilbert space, where for each \(\gamma\), \(s(=s_\gamma)\) is a unit speed reparametrization.
\end{dfn}

It is clear that \(C_0^\infty(\co)\subset H_\Gamma^{1,0}(\co)\), but the question is whether it is dense in this new space. Before that, we need to make several assumptions on the pattern. Notice that there are in general four cases in viewing the topological components in \(\Gamma\): regular closed curves, arcs without intersecting others, closed curves with a conic point and graphs composed of arcs and intersection points. We will call the 
conic/intersection/end points arising in all these cases ``nodes".

\begin{dfn}
A pattern is called \textbf{regular} if 
\begin{itemize}
    \item near each node of the pattern, the arcs that approach this node are mutually nontangential;
    \item any arc that intersects the outer boundary \(\po\) will approach it nontangentially.
\end{itemize}\label{regularpattern}
\end{dfn}

\begin{thm}{\emph{(\textbf{Density Result})}}
If the pattern \(\mathcal{G}\) is regular, then the subspace \(C_0^\infty(\co)\) is indeed dense in \(H_\Gamma^{1,0}(\co)\).\label{densityresult}
\end{thm}

To show this theorem, several basic facts have to be clarified.
\begin{lem}
Suppose that a pattern \(\mathcal{G}\) is regular, and in \(\Gamma\) there is only one element \(\gamma\) that does not intersect the outer boundary \(\po\). If either of
\begin{itemize}
    \item[1.] \(\gamma\) is a \(C^2\) closed curve;
    \item[2.] \(\gamma\) has two different end points,
\end{itemize}
is satisfied, then, without any regularity assumption on \(\po\), we have the above density result.
\end{lem}
\begin{proof}
The idea is to split \(u\in H_{\Gamma}^{1,0}(\co)\) into a sum \(\mu+\nu\), where \(\mu\equiv0\) on \(\gamma\). In case 1, we extend \(\restr{u}{\gamma}\) to a tubular neighborhood \(N_\delta=\{\x\in\co;dist(\x,\gamma)<\delta\}\) for some \(\delta>0\) by taking constant values along the normals. The auxiliary function \(\nu\) is then defined to be the multiplication of the extension and a cut-off function \(\eta\in C_0^1(\co)\) satisfying \(\eta\equiv1\) on \(N_{\delta/2}\) and \(\eta\equiv0\) on \(N_{3\delta/4}\). Because \(\restr{u}{\gamma}\in H^1(\gamma)\), we know that \(\nu\in H_0^1(\co)\). Let \(\nu_n\in C^1(\gamma)\) be a sequence that converges to \(\restr{u}{\gamma}\) in \(H^1(\gamma)\) as \(n\rta\infty\), then the extensions of \(\nu_n\) similar to \(\restr{u}{\gamma}\) will then converge to \(\nu\) in \(H_\Gamma^{1,0}(\co)\) as \(n\rta\infty\). Since \(\gamma\) is \(C^2\), and \(\restr{\mu}{\gamma}\equiv0\), we know that, according to classical trace theorem, \(\mu\) can be approximated in \(H_0^1(\co)\) by a sequence of \(C_0^\infty(\co)\) functions with restriction 0 near \(\gamma\).

Case 2 is similar to the first one. We first extend \(\gamma\) a little bit along the tangent lines at the two end points, so that the extension \(\Tilde{\gamma}\) is still a \(C^1\) curve of case 2. Since \(\restr{u}{\gamma}\in H^1(\gamma)\), we know by Sobolev Imbedding Theorem, it is also in \(C^{1/2}(\gamma)\). From \(\gamma\) to \(\Tilde{\gamma}\), we extend \(\restr{u}{\gamma}\) by connecting boundary values and 0 linearly, so that the extension \(\Tilde{u}\in H_0^1(\Tilde{\gamma})\). Similarly to Case 1, we may further extend \(\Tilde{u}\) to be some \(\nu\in H_0^1(\co)\). The remaining problem is that whether there is a sequence of \(C_0^\infty(\co)\) functions with restriction 0 near \(\gamma\) approximating \(\mu=u-\nu\) in \(H_0^1(\co)\).

To tackle this problem, we assume that one of the end point is at the origin, and for \(k>0\), we consider
\[
\mu_k(\theta,r)=\mu(\theta,r)\cdot\psi(k r),
\]
where \((\theta,r)\in[-\pi,\pi]\times[0,\infty)\) is the polar coordinate system, and \(\psi\in C_0^{\infty}(\R_+)\) satisfying \(\psi(r)\equiv0, \,r\le1\), \(\psi(r)\equiv1,\,r\ge2\) and \(0\le\psi\le1\). Clearly \(\mu_k\) converges to \(\mu\) in \(L^2(\co)\) as \(k\rta\infty\), and if one can prove that \(\gd\mu_k\) also converges to \(\gd\mu\) in \(L^2(\co)\), then the approximation problem is reduced to a piecewise \(C^1\) (referring to the interior boundary) domain \(\co\backslash\lb B(p,1/k)\cup B(q,1/k)\cup\gamma\rb\), which is a known result \citep{Mazya-2010}.

Observe that
\[
\lw\gd\mu_k\rw^2=|(\mu_k)_r|^2+\frac{1}{r^2}|(\mu_k)_\theta|^2.
\]
Clearly, \((\mu_k)_\theta=(\mu\psi(k r))_\theta=(\mu)_\theta\psi(k r)\), and so we only have to worry about \((\mu_k)_r\). Notice that
\[
(\mu\psi)_r=\mu_r\psi+k\mu\psi'(k r),
\]
and thus it suffices to evaluate
\be
   \int_{\co} k^2 (\mu \psi'(k r))^2 r d r d\theta\le C k^2 \int_{-\pi}^{\pi}\int_{1/k}^{2/k} \mu^2 r d r d\theta.\label{balltruncationerror}
\ee
Observing that, when \(k\) is large, any component of \(\gamma\) contained in the ring \(\{1/k<r<2/k\}\) is nontangential to the \(\theta\) direction of the ring, we may write this component as some function graph \(\{\theta=g(r);1/k<r<2/k\}\). By the proof of Lemma (\ref{rotationfact}), we know that the error (\ref{balltruncationerror}) satisfies
\[
\begin{split}
    RHS&=C k^2 \int_{-\pi}^{\pi}\int_{1/k}^{2/k} (\mu(r,\theta)-\mu(r,g(r)))^2 r d r d\theta\\
       &\le C' k^2\int_{1/k}^{2/k}\lb\int_{-\pi}^{\pi}|\mu_\theta|d\theta\rb^2 r d r\\
       &\le  C'' k^2\int_{1/k}^{2/k}r^2\int_{-\pi}^{\pi}\frac{|\mu_\theta|^2}{r^2}r d\theta d r\\
       &\le C'''\int_{B(\mathbf{0},2/k)} |\gd \mu|^2 d\x\\
       &\overset{k\rta\infty}{\longrightarrow}0.
\end{split}
\]

\end{proof}

\begin{lem}
Suppose that a pattern \(\mathcal{G}\) is regular, and in \(\Gamma\) is composed of exactly one arc \(\gamma\) that has one end contained in \(\co\) and another on \(\po\) at the origin. Then if \(u\in H_\Gamma^{1,0}(\co)\) satisfies \(\restr{u}{\gamma}(\mathbf{0})=0\), we show that \(u\) can be approximated by \(C_0^\infty(\co)\) functions in \(H_\Gamma^{1,0}(\co)\).
\end{lem}

\begin{proof}
The crucial point is to delete a proper function so that \(u\) vanishes on \(\gamma\). To do this, we start with a small disc \(B(\mathbf{0},d)\) with \(d>0\) and consider a smooth cut-off of \(u\) in this disc. Since \(\mathcal{G}\) is regular, \(\gamma\) approaches \(\po\) nontangentially, and so we may find \(\theta_1,\theta_2\in(-\pi,\pi)\) and \(\epsilon>0\) such that the cone \(\{\theta_1+\epsilon<\theta<\theta_2-\epsilon\}\) contains \(\gamma\), and the cone \(\{\theta_1<\theta<\theta_2\}\cap B(\mathbf{0},d)\subset\co\). Then we define \(0\le\phi(\theta)\le1\) to be the cut-off function on \((-\pi,\pi)\) such that \(\phi\equiv0\) outside \((\theta_1,\theta_2)\), and \(\phi\equiv1\) in \((\theta_1+\epsilon,\theta_2-\epsilon)\).

When \(d>0\) is small enough, the segment of \(\gamma\) in the disc can be written as some function graph \(\{\theta=g(r);0\le r<d\}\). Since \(\gamma\) is \(C^2\), we know that \(\mu(r)\coloneqq u(r,g(r))\) is in \(H^1(0,d)\). We claim that \(\phi(\theta) \mu (r)\) is in \(H_0^1(\{\theta_1<\theta<\theta_2\})\). To see this, we merely have to worry about the \(L^2\) integrability of the gradient. Since \((\mu(r)\phi(\theta))_r=\mu_r(r)\phi(\theta)\), we only have to check the derivative along \(\theta\), which is
\[
\begin{split}
    \int_0^{d}\int_{\theta_1}^{\theta_2}\frac{\lw \mu(r)\phi'(\theta)\rw^2}{r^2} \cdot r d \theta d r&\le C (\theta_2-\theta_1) \int_0^{d} (\mu(r))^2/r d r\\
    &\le C'(\theta_2-\theta_1) d\\
    &<\infty,
\end{split}
\]
where we have used the embedding \(H^1(0,d)\hookrightarrow C^{1/2}([0,d])\), which implies that \(\lw \mu(r)-\mu(0)\rw\le r^{1/2}\). Moreover, it is not hard to observe that \(\mu\phi\) can be approximated in \(H_{\gamma}^{1,0}\lb\{\theta_1\le\theta\le\theta_2\}\rb\) by \(C_0^\infty\lb\{\theta_1<\theta<\theta_2\}\rb\) functions.

Now, by a partition of unity and compactness of \(\gamma\), we may handle the trace piece by piece, and finally we arrive at the piece that involves another end point. This end point can be dealt with by using both the above method and the extension method in the proof of the above lemma. 

\end{proof}

\begin{lem}
Suppose that a pattern \(\mathcal{G}\) is regular, and in \(\Gamma\) there is only one component compactly contained in \(\co\), composed of only one node at the origin and arcs \(\gamma_1,\cdots,\gamma_m\in\mathcal{G}\) connecting to it. Then if \(u\in H_\Gamma^{1,0}(\co)\) satisfies
\[
\restr{u}{\gamma_i}(\mathbf{0})=\restr{u}{\gamma_j}(\mathbf{0}),\,\forall 1\le i\le j\le m,
\]
we show that \(u\) can be approximated by \(C_0^\infty(\co)\) functions in \(H_\Gamma^{1,0}(\co)\).\label{raydensity}
\end{lem}

\begin{proof}
Deleting a \(C_0^\infty(\co)\) function that takes value \(\restr{u}{\gamma_1}(\mathbf{0})\) at the origin, we see that we only need consider \(u\) that has value 0 at the origin. By a partition of unity, we may safely consider \(u\) with support contained in the disc \(B(\mathbf{0},d)\) for some small \(d>0\), and we have the following observations:
\begin{itemize}
    \item[1.] By the property of the partition of unity, we may without losing generality assume that \(\restr{u}{\gamma_{i}}=0\) on \(\partial B(\mathbf{0},d)\) for each \(i=1,\cdots,m\);
    \item[2.] Because \(d>0\) is chosen small, we see that each arc \(\gamma_i\) can be written as some function graph \(\{\theta=g_i(r);0\le r <d\}\) in polar coordinates.
\end{itemize}

By the second observation, each \(\restr{u}{\gamma_i}\) can be extended to \(B(\mathbf{0},d)\) by revolution, and we denote this extension by \(\mu_i\). Suppose both \(\{\theta_1\le \theta\le \theta_2\}\) and \(\{\theta_1+\epsilon\le \theta\le \theta_2-\epsilon\}\) for \(\epsilon>0\,,\theta_{1},\theta_2\in(-\pi,\pi)\) are cones that only contain \(\gamma_i\). Then we define \(0\le\phi_i(\theta)\le1\) to be the cut-off function defined in the proof of the above lemma, and by the proof of the above lemma, \(\mu_i\phi_i\in H_0^1\lb\{\theta_1\le\theta\le\theta_2\}\rb\) and \(\mu_i\phi_i\) can be approximated in \(H_{\gamma_i}^{1,0}\lb\{\theta_1\le\theta\le\theta_2\}\rb\) by \(C_0^\infty\lb\{\theta_1<\theta<\theta_2\}\rb\) functions.

Therefore, we see \(u-\sum_{i=1}^m\mu_i\phi_i\) is an \(H_0^1(B(\mathbf{0},d))\) function with trace \(0\) on \(\Gamma\) (this is more exactly ``\(
\Gamma\cap B(\mathbf{0},d)\)", but since we have reduced our problem to this special case, we use the same notation), which can certainly be approximated by \(C_0^\infty(B(\mathbf{0},d))\) functions that are 0 near \(\Gamma\).

\end{proof}

\begin{lem}
The conditions on \(\mathcal{G}\) in the above lemma implies that for every \(u\in H_\Gamma^{1,0}(\co)\)
\[
\restr{u}{\gamma_i}(\mathbf{0})=\restr{u}{\gamma_j}(\mathbf{0}),\,\forall 1\le i\le j\le m.
\]
\end{lem}

\begin{proof}
We set \(h_i=\restr{u}{\gamma_i}(\mathbf{0})\) for each \(i=1,\cdots,m\), and consider some \(\eta\in C_0^\infty(\co)\) with \(\eta(\x)\equiv1\) for \(\x\in B(\mathbf{0},d/2)\). Define \(\eta_i=h_i\eta\). For each \(i\), we have \(\restr{(u-\eta_i)}{\gamma_i}(\mathbf{0})=0\), and so by the proof of Lemma (\ref{raydensity}), we may find some \(H_0^1(B(\mathbf{0},d))\) function \(\nu_i\) with support disjoint from other curves so that \(\restr{(u-\eta_i-\nu_i)}{\gamma_i}\equiv0\). Therefore, we may observe that \(\nu\coloneqq u-\sum_{i=1}^m \nu_i\) is constant \(h_i\) on each \(\gamma_i\cap  B(\mathbf{0},d/2)\).

Recall that when \(d\) is small enough, each segment \(\gamma_i\cap  B(\mathbf{0},d)\) can be written as function graph \(\{\theta=g_i(r);1\le r<d\}\), and for convenience, we assume that \(g_i\le g_j\) for \(1\le i\le j\le m\). Taking \(i=1\) as an example, we consider for \(k>1\) the domain \(O_k = \{g_1(r)<\theta<g_2(r); d/(k+2)<r<d/2\}\). Notice that by Lemma (\ref{raydensity}), there will be a sequence \(v^n\in C^1(O_k)\cap C^0(\overline{O_k})\) that converges to \(\nu\) in \(H^1(O_k)\) and satisfies \(\restr{v^n}{\{\gamma_i;d/(k+2)<r<d/2\}}\equiv h_i\).

Applying Fundamental Theorem of Calculus to \(v^n\), we have for \(d/(k+2)<r<d/2\),
\[
\int_{g_1(r)}^{g_2(r)} v_\theta^n (r,\theta) d\theta = h_2-h_1.
\]
If \(h_2\ne h_1\), then we have the inequality
\[
(g_2(r)-g_1(r))\int_{g_1(r)}^{g_2(r)} (v_\theta^n)^2 d\theta \ge |h_2-h_1|^2,
\]
and because \(g_2(r)>g_1(r)\) for \(r>0\),
\[
\int_{g_1(r)}^{g_2(r)} (v_\theta^n)^2 d\theta \ge \frac{|h_2-h_1|^2}{(g_2(r)-g_1(r))}>\frac{|h_2-h_1|^2}{2\pi}.
\]
Moreover, we have
\[
\begin{split}
    \int_{O_k}|\gd v^n|^2 d\x&\ge \int_{d/(k+2)}^{d/2}\int_{g_1(r)}^{g_2(r)}\frac{|v_\theta^n|^2}{r^2}\cdot r d\theta d r\\
                         &\ge \frac{|h_2-h_1|^2}{2\pi}\int_{d/(k+2)}^{d/2} \frac{1}{r} d r\\
                         &\ge O(|\ln{(k)}|),
\end{split}
\]
and so by the convergence of \(v^n\) to \(\nu\) in \(H^1(O_k)\), we have the estimate
\[
\int_{O_k} |\gd \nu|^2 d\x \ge  O(|\ln{(k)}|),
\]
which is impossible because \(\nu\in H^1(B(\mathbf{0},d))\).

\end{proof}

\begin{cor}
If a curve \(\gamma\subset \co\) approaches \(\po\) at the origin nontangentially, then for \(u\in H_\gamma^{1,0}(\co)\), we have \(\restr{u}{\gamma}(\mathbf{0})=0\).
\end{cor}
\begin{proof}
The proof is covered by the one of the above lemma.
\end{proof}

\begin{proof}[Proof of Theorem (\ref{densityresult})]
By a partition of unity, the problem can be divided into the cases that have been discussed above in detail.
\end{proof}

\subsection{EBCs on Patterns: Truncation Method}
In this subsection, we deal with an EBC problem on some Lipschitz domain \(\co\) associated with a given pattern \(\mathcal{G}\). The formation of an EBC in some ``distributional" sense does not require the regularity of the whole pattern described in prior subsection. However, the regularity of the whole pattern ensures that the solution to the effective model is unique. 

Recalling in prior subsection, we have already decomposed \(\mathcal{G}\) into several cases. In this subsection, we also keep in mind these cases, but because the methods in dealing with different patterns are so similar (we call it \emph{truncation}), we only present a proof for a classical case: \(\mathcal{G}\) composed of one arc \(\gamma=\Gamma\) with one end on \(\po\) and one in \(\co\). 

Let \(\delta>0\) be small and \((s,\tau),\,s\in(0,l),\,\tau\in(-\delta,\delta)\) a normal reparametrization of a \(\delta\)-tubular neighborhood of \(\gamma\) defined in Introduction. As before, we define \(R_\delta\) to be the intersection of \(\co\) and the image of \((0,l)\times(-\delta,\delta)\) in \(\R^2\). We further define the truncated roads \(R_{\delta,q}\) to be the image of \((q,l-q)\times(-\delta,\delta)\) for some \(q>0\). By basic geometry, for every \(q>0\), there is some \(\delta_q>0\) such that for all \(0<\delta<\delta_q\), the truncated road \(R_{\delta, q}\csubset \co\).

What we concern about is the following problem
\be
\bca
-\gd\cdot(\bar{\sigma}\gd u^\delta) (\x)=f(\x),&\x\in\co,\\
u^\delta(\x)=0,&\x\in\po,
\eca\label{intcapoutoriginalequation}
\ee
where \(\bar{\sigma}=\sigma>>1\) in \(R_\delta\cap\{\tau>0\}\) and \(=1\) elsewhere. Assuming that \(\sigma\delta\rta a>0\), what will \(u^\delta\) converge to? This problem has been well investigated by H. Li and X. Wang \citep{Li-2017,Li2020} for the case that \(\Gamma\) is a circle without conic points that is contained in the domain \(\co\). In fact, their proofs are the main tools to tackle this generalization.

We start with the weak solution \(u^\delta\in H_0^1(\co)\) to (\ref{intcapoutoriginalequation}), and then for \(v\in H_0^1(\co)\), 
\be
\int_\co \gd v\cdot(\bar{\sigma}\gd u^\delta)=\int_\co fv,
\ee
and by replacing \(v=u^\delta\), we obtain the energy estimate for \(u^\delta\)
\be
\int_\co \bar{\sigma} |\gd u^\delta|^2\le O(1),
\ee
which leads to the weak convergence of \(u^\delta\) to some \(u^\ast\) in \(H_0^1(\co)\), as \(\delta\rta0\).

\begin{lem}
Suppose that \(V_1\in\po,V_2\in\co\) are the two end points of \(\Gamma\). Then for every \(\co'\csubset\co\backslash\{V_2\}\) and \(d=\min\{dist(\co',\po),dist(\co',V_2)\}>0\), there is some \(\delta_d>0\) such that for all \(0<\delta<\delta_d\), we have
\[
\int_{\co'} \bar{\sigma} |\gd^2 u|^2 \le O(1).
\]\label{D2estimate}
\end{lem}

\begin{proof}Assume that \(l>0\) is the arc-length of \(\Gamma\) and define \(d/6>\delta_d>0\) to be the number such that for all \(0<\delta<2\delta_d\), the truncated road \(\displaystyle R_{\delta,\frac{d}{3l}}\) is compactly contained in \(\co\backslash\{V_2\}\). For convenience, we set \(p=\frac{3}{2}\delta_d\) and \(q=\frac{d}{2l}\). Observe that \(R_\delta\cap\co'\subset R_{\delta,q}\) for \(0<\delta<\delta_d\), then it suffices to establish the estimate in the region \(\displaystyle R_{p,q}\), because the operator coefficients are smooth elsewhere.

Now, by the choice of \(\delta_d\), we have that the mapping \((s,\tau)\mapsto\Gamma(s)+\tau\vec{n}(s)\) defines a diffeomorphism from \((h+2q/3,k-2q/3)\times(-4p/3,4p/3)\) to \(R_{4p/3,2q/3}\). On \(R_{4p/3,2q/3}\), assuming that \(0<\delta<\delta_d\), one may derive from the PDE of (\ref{intcapoutoriginalequation}) that
\be
(1+\tau \kappa(s))f +\bar{\sigma}\lb\frac{u_s}{1+\tau\kappa(s)}\rb_s+\bar{\sigma}\lb(1+\tau\kappa(s))u_\tau\rb_\tau=0,\label{curvilinearform}
\ee
where \(\kappa\) is the curvature of \(\Gamma\). The above equation holds locally for \(s\in(2q/3,l-2q/3)\), and \(\tau\in(-4p/3,0)\), \((0,\delta)\) and \((\delta,4p/3)\) respectively. We also have the \emph{transmission condition}
\be
\bca
u(s,0^-)=u(s,0^+),&u_\tau(s,0^-)=\sigma u_\tau(s,0^+),\\
u(s,\delta^-)=u(s,\delta^+),& u_{\tau}(s,\delta^+)=\sigma u_{\tau}(s,\delta^-).
\eca\label{transmissioncondition}
\ee

We define two test functions
\[
\eta(\tau)=\bca
1,& |\tau|\le \delta_d,\\
0,& |\tau|>p,
\eca
\]
and
\[
\xi(s)=\bca
1,& s\in (q,l-q),\\
0,& s\in (2q/3,5q/6)\cup(l-5q/6,l-2q/3),
\eca
\]
where both \(\xi\) and \(\eta\) are bounded by 0 and 1. Applying \(\eta^2\xi^2 u_{ss}\) to (\ref{curvilinearform}), we obtain
\be
\int_{I}\eta^2\xi^2 u_{ss}(1+\tau \kappa(s))f  +\bar{\sigma}\lb\frac{u_s}{1+\tau\kappa(s)}\rb_s\eta^2\xi^2 u_{ss}+\bar{\sigma}\lb(1+\tau\kappa(s))u_\tau\rb_\tau\eta^2\xi^2 u_{ss}d\tau ds=0,
\ee
where \(I=(2q/3,l-2q/3)\times(-4p/3,4p/3)\). Notice that
\[
\begin{split}
   \int_I \bar{\sigma}\lb(1+\tau\kappa(s)u_\tau)\rb_\tau\eta^2\xi^2 u_{ss}d\tau ds&=
-\int_I \bar{\sigma}\lb(1+\tau\kappa(s)u_\tau)\rb_{\tau s}\eta^2\xi^2 u_{s}d\tau ds\\
&\quad-2\int_I \bar{\sigma}\lb(1+\tau\kappa(s)u_\tau)\rb_\tau\eta^2\xi\xi_s u_{s}d\tau ds\\
&\eqqcolon -Y_1-2Y_2.
\end{split}
\]
Then
\[
\begin{split}
  Y_1&=\int_{2q/3}^{l-2q/3}\int_{-4p/3}^{0} \lb(1+\tau\kappa(s))u_\tau\rb_{\tau s}\eta^2\xi^2 u_{s}d\tau ds\\ &\quad+\int_{2q/3}^{l-2q/3}\int_{0}^{\delta} \sigma\lb(1+\tau\kappa(s))u_\tau\rb_{\tau s}\eta^2\xi^2 u_{s}d\tau ds\\ &\quad+\int_{2q/3}^{l-2q/3}\int_{\delta}^{4p/3} \lb(1+\tau\kappa(s))u_\tau\rb_{\tau s}\eta^2\xi^2 u_{s}d\tau ds\\
  &\overset{(\ref{transmissioncondition})}{=}\int_{2q/3}^{l-2q/3}\int_{-4p/3}^{0} \lb(1+\tau\kappa(s))u_\tau\rb_{s}\xi^2\lb\eta^2 u_{s}\rb_\tau d\tau ds\\ &\quad+\int_{2q/3}^{l-2q/3}\int_{0}^{\delta} \sigma\lb(1+\tau\kappa(s))u_\tau\rb_{s}\xi^2\lb\eta^2 u_{s}\rb_\tau d\tau ds\\ &\quad+\int_{2q/3}^{l-2q/3}\int_{\delta}^{4p/3} \lb(1+\tau\kappa(s))u_\tau\rb_{s}\xi^2\lb\eta^2 u_{s}\rb_\tau d\tau ds\\
  &=\int_{2q/3}^{l-2q/3}\int_{-4p/3}^{4p/3} \bar{\sigma}\lb(1+\tau\kappa(s))u_\tau\rb_{s}\xi^2\lb\eta^2 u_{s}\rb_\tau d\tau ds,
\end{split}
\]
and
\[
\begin{split}
    Y_2&=\int_{2q/3}^{l-2q/3}\int_{-4p/3}^{0} \lb(1+\tau\kappa(s))u_\tau\rb_\tau\eta^2\xi\xi_s u_{s}d\tau ds\\ &\quad+\int_{2q/3}^{l-2q/3}\int_{0}^{\delta} \sigma\lb(1+\tau\kappa(s))u_\tau\rb_\tau\eta^2\xi\xi_s u_{s}d\tau ds\\ &\quad+\int_{2q/3}^{l-2q/3}\int_{\delta}^{4p/3} \lb(1+\tau\kappa(s))u_\tau\rb_\tau\eta^2\xi\xi_s u_{s}d\tau ds\\
     &\overset{(\ref{transmissioncondition})}{=}\int_{2q/3}^{l-2q/3}\int_{-4p/3}^{0} \lb(1+\tau\kappa(s))u_\tau\rb\xi\xi_s\lb\eta^2 u_{s}\rb_\tau d\tau ds\\ &\quad+\int_{2q/3}^{l-2q/3}\int_{0}^{\delta} \sigma\lb(1+\tau\kappa(s))u_\tau\rb\xi\xi_s\lb\eta^2 u_{s}\rb_\tau d\tau ds\\ &\quad+\int_{2q/3}^{l-2q/3}\int_{\delta}^{4p/3} \lb(1+\tau\kappa(s))u_\tau\rb\xi\xi_s\lb\eta^2 u_{s}\rb_\tau d\tau ds\\
  &=\int_{2q/3}^{l-2q/3}\int_{-4p/3}^{4p/3} \bar{\sigma}\lb(1+\tau\kappa(s))u_\tau\rb\xi\xi_s\lb\eta^2 u_{s}\rb_\tau d\tau ds.
\end{split}
\]
The rest of the proof is covered by that of Lemma (3.2) in \citep{Li2020}.
\end{proof}

\begin{lem}
We have that the limit \(u^\ast\in H_\Gamma^{1,0}(\co)\), and
\be
\int_\co |\gd u^\ast|^2+a\int_{\Gamma} (u_s^\ast)^2\lesssim_{\co} \int_\co f^2.
\ee
\label{L2gammalemma}
\end{lem}
\begin{proof}
Recall that we have
\be
\int_{\co} \bar{\sigma} |\gd u^\delta|^2 \lesssim_{\co} \int_{\co} f^2.\label{ebcenergyestimateinproof}
\ee
Because \(u^\delta\) weakly converges to \(u^\ast\) in \(H_0^1(\co)\), we have
\be
\liminf_{\delta\rta0}\int_{\co} |\gd u^\delta|^2 \ge \int_\co |\gd u^\ast|^2.\label{ebcenergyestimateinproof1}
\ee
Moreover, we have
\be
    (\sigma-1)\int_{R_\delta} |\gd u^\delta|^2 \gtrapprox \sigma\int_{q}^{l-q}\int_0^\delta (u_s^\delta)^2+(u_\tau^\delta)^2 d\tau ds, \label{interiorboundarytermestimate}
\ee
with \(0<\delta<\delta_q\). According to Lemma (\ref{D2estimate}), we have
\[
\begin{split}
    \sigma\lw\int_{q}^{l-q}\int_0^\delta (u_s^\delta)^2(s,\tau)-(u_s^\delta)^2(s,0^+) d\tau ds\rw&\le \sigma\lw\int_{q}^{l-q}\int_0^\delta (u_s^\delta(s,\tau)-u_s^\delta(s,0^+))^2 d\tau ds\rw\\ &\quad+\sigma\lw\int_{q}^{l-q}\int_0^\delta2u_s^\delta(s,0^+)(u_s^\delta(s,\tau)-u_s^\delta(s,0^+)) d\tau ds\rw\\
    &= Q_1+Q_2, 
\end{split}
\]
where 
\[
\begin{split}
    Q_1&\le \sigma\lw\int_{q}^{l-q}\int_0^\delta \lb\int_0^\delta u_{s\tau}^\delta d\tau\rb^2 d\tau ds\rw\\
       &\le O(\delta^{3/2}),
\end{split}
\]
and
\[
\begin{split}
    Q_2&\le 2\sigma\lb\int_{q}^{l-q}\int_0^\delta (u_s^\delta(s,0^+))^2 d\tau ds\rb^{1/2} \lb\int_{q}^{l-q}\int_0^\delta(u_s^\delta(s,\tau)-u_s^\delta(s,0^+))^2  d\tau ds\rb^{1/2}\\
    &\le C \delta^{1/2} \lb\sigma^2\int_{q}^{l-q}\int_0^\delta\lb\int_0^\delta u_{s\tau}^\delta d\tau\rb^2  d\tau ds\rb^{1/2}\\
    &\le O(\delta).
\end{split}
\]
Here we have used trace theorem. Similar things still occur if we replace \(u_s^\delta\) by \(u_\tau^\delta\) in (\ref{interiorboundarytermestimate}), and thus, we have
\[
 (\sigma-1)\int_{R_\delta} |\gd u^\delta|^2 \gtrapprox \sigma\delta\int_{q}^{l-q} (u_s^\delta)^2(s,0)ds+\delta\int_{q}^{l-q}(u_\tau^\delta)^2(s,0^-) d\tau ds +o(1).
\]
Because by Lemma (\ref{D2estimate}), \(u^\delta\) converges weakly to \(u^\ast\) in \(H_0^1(\co')\) for any \(\co'\csubset\co\) without crossing \(\Gamma\), according to trace theorem, \(u^\delta\) converges weakly to \(u^\ast\) in \(H^1(\overline{\co'}\cap\Gamma)\), which leads to the following inequality
\be
\liminf_{\delta\rta0} \sigma\delta\int_{q}^{l-q} (u_s^\delta)^2(s,0)ds  \ge a\int_{q}^{l-q} (u_s^\ast)^2(s,0)ds.\label{ebcenergyestimateinproof2}
\ee
Combining (\ref{ebcenergyestimateinproof}), (\ref{ebcenergyestimateinproof1}) and (\ref{ebcenergyestimateinproof2}), we have
\[
\int_\co |\gd u^\ast|^2+a\int_{q}^{l-q} (u_s^\ast)^2ds\lesssim_{\co} \int_\co f^2,
\]
and by using monotone convergence theorem, we are done.
\end{proof}

\begin{lem}
For every \(\eta\in C_0^\infty(\co)\), we have the convergence
\be
\int_\co \bar{\sigma}\gd u\cdot\gd\eta \overset{\delta\rta0}{\longrightarrow} \int_{\co} \gd u^\ast\cdot \gd\eta + a \int_{\Gamma} u_s^\ast \eta_s.\label{2.14}
\ee
\end{lem}

\begin{proof}
It suffices to consider the limit of the following quantity for \(n>1\) and \(0<\delta<\delta_{1/n}\)
\[
\begin{split}
    \sigma\int_{R_\delta} \gd u\cdot \gd\eta&=\sigma\int_{R_{\delta,1/n}} \gd u\cdot \gd\eta + O\lb\sigma\int_{R_{\delta}\cap B(V_2,3/n)}|\gd u||\gd\eta|\rb\\
                                        &= I_1 +I_2.
\end{split}
\]
Notice that \(I_1\) converges to \(a\int_{h+1/n}^{k-1/n} u_s^\ast\eta_s\) as \(\delta\rta0\) by Lemma (\ref{D2estimate}) and transmission conditions. Be careful that this part requires the \(C^2\) smoothness of the arc.

Moreover, the error term satisfies
\[
\begin{split}
I_2 &\lesssim \sigma\lb\int_{R_{\delta}\cap B(V_2,3/n)} |\gd u|^2\rb^{1/2} \lb\int_{R_{\delta}\cap B(V_2,3/n)} |\gd \eta|^2\rb^{1/2}\\
    &\lesssim \lb\sigma\delta\cdot6/n \rb^{1/2} \norm{\eta}_{C_0^1(\co)}\\
    &=O\lb\sqrt{\frac{1}{n}}\rb.
\end{split}
\]
Using Lemma (\ref{L2gammalemma}), we know, by Lebesgue Dominated Convergence Theorem, that after sending \(n\rta\infty\), \(a\int_{1/n}^{l-1/n} u_s^\ast\eta_s\) will converge to \(a\int_\Gamma u_s^\ast\eta_s\), and in effect, we obtain the convergence (\ref{2.14}).
\end{proof}

\noindent\textbf{Remark}: The setting for a general pattern is quite similar to the above one. Associating each arc \(\gamma\in\mathcal{G}\) with a road family \(R_\delta^{\gamma}\) as defined before, we consider the union \(R_\delta(=R_\delta^{\mathcal{G}})=\cup_{\gamma\in\mathcal{G}}R_\delta^{\gamma} \), and call this a \emph{road net} with width \(\delta\). The \emph{truncated} road net will be simply the union of truncated roads. The truncation method used in the case with one arc having end points one on \(\po\) and one in \(\co\) above can be immediately applied to other cases with different patterns. Without a regularity requirement on the pattern, the following convergence result is always true, provided that each arc is \(C^2\):
\[
\sigma\int_{R_\delta^{\mathcal{G}}\cap\{\tau>0\}}  \gd u\cdot\gd\eta\overset{\delta\rta 0}{\longrightarrow} a \sum_{\gamma\in\mathcal{G}}\int_{\gamma} u_s^\ast \eta_s,
\]
where \(u^\ast\) is the weak limit of \(u\) in \(H_0^1(\co)\) and \(\eta\in C_0^\infty(\co)\). (The tangential derivative of \(u^\ast\) on the arc is well-defined because of Lemma (\ref{D2estimate}))

\subsection{Specifications on the EBCs on Patterns}

By the above subsection, we obtain an integral equation for the limit \(u^\ast\in H_\Gamma^{1,0}(\co)\), that
\be
\int_{\co} \gd u^\ast\cdot\gd\eta +a \sum_{\gamma\in\mathcal{G}} \int_{\gamma} u_s^\ast\eta_s =\int_\co f\eta,\label{distributionalintegralequation}
\ee
for all \(\eta\in C_0^\infty(\co)\). This result does not require the pattern \(\mathcal{G}\) to be regular. 

By the Density Result, we know that if the pattern is regular, then the above equation holds for all \(\eta\in H_\Gamma^{1,0}(\co)\), which gives both the uniqueness and existence for \(u^\ast\). We claim that in this situation the PDE of \(u^\ast\) takes the form
\be
\bca
-\Delta u^\ast (\x) = f (\x), &\x\in\co\backslash\Gamma,\\
u^\ast(\x) \text{ satisfies \emph{Road Pattern Condition}},&\x\in \Gamma,\\
\restr{u}{\po}\equiv0,&
\eca\label{theebcpde}
\ee
where the Road Pattern Condition is described in the following list:
\begin{itemize}
    \item[1.] In the interior of each arc \(\gamma\in\mathcal{G}\), if \(s\) represents a unit speed representation and \(\vec{n}\) a unit normal field on \(\gamma\), then locally the domain is split into two by the arc, and we may call the one directed by \(\vec{n}\) positive, and another negative. In this arc, \(\restr{u}{\gamma}\) pointwise satisfies \(u^+=u^-\) and \(a u_{ss}= \frac{\partial u^-}{\partial\vec{n}}-\frac{\partial u^+}{\partial\vec{n}}\);
    \item[2.] \(\restr{u}{\gamma}(V)=0\) for all \(V\in\gamma\cap\po\);
    \item[3.] If \(V\) is contained in \(\co\), and \(\gamma,\gamma'\in\mathcal{G}\) are two arcs that end at it, then \(\restr{u}{\gamma}(V)=\restr{u}{\gamma'}(V)\);
    \item[4.] For \(V\in\co\), if \(\mathcal{G}_V\) is the collection of arcs that end at it, then we have
    \[
    \sum_{\gamma\in\mathcal{G}_V} u_{s_\gamma}(V)=0,
    \]
    where \(s_\gamma\) is the arc-length parametrization starting from \(V\);
    \item[5.]According to elementary geometry, by introducing several ``ghost" curves, the pattern \(\mathcal{G}\) (if it is regular, but it still hold even if it is not regular near nodes on the outer boundary \(\po\)) will split a \(C^1\) subdomain \(\co'\csubset \co\) into several piecewise \(C^1\) subdomains. On these ghost curves, the solution is considered to satisfy ``no-effect" condition: \(u^+=u^-\), and \(\frac{\partial u^-}{\partial\vec{n}}\equiv\frac{\partial u^+}{\partial\vec{n}}\). By classical elliptic theory, we know that such boundary condition has no effect, and the solution is simply smooth across the corresponding boundary, which implies that the choice of ``ghost" curves is not intrinsic in studying the problem.
\end{itemize}

Let \(v\) be some piecewise \(C^2\) and continuous function on \(\co\) that satisfies (\ref{theebcpde}). Multiplying \(\eta\in C_0^\infty(\co)\) to \(-\Delta v=0\), and do integration by parts on each subdomain obtained by splitting \(supp\{\eta\}\csubset\co_\eta\csubset\co\) using \(\mathcal{G}\) and some ``ghost" curves described in 5.. This will give
\[
\int_{\co} \gd v\cdot\gd\eta - \sum_{\gamma\in\mathcal{G}\cup\{\text{``Ghost" curves}\}}\int_{\gamma} \eta\lb\frac{\partial v^-}{\partial\vec{n}}-\frac{\partial v^+}{\partial\vec{n}}\rb= \int_\co f \eta.
\]
Noticing that the second term in the left hand side equals
\[
\begin{split}
   \sum_{\gamma\in\mathcal{G}} \int_{\gamma} \eta\lb\frac{\partial v^-}{\partial\vec{n}}-\frac{\partial v^+}{\partial\vec{n}}\rb&= a \sum_{\gamma\in\mathcal{G}}\int_{\gamma} v_{ss} \eta\\
   &=-a \sum_{\gamma\in\mathcal{G}}\lb\int_{\gamma} v_{s} \eta_s + v_s(V_2^\gamma)\eta(V_2^\gamma)- v_s(V_1^\gamma)\eta(V_1^\gamma)\rb\\
   &=-a\sum_{\gamma\in\mathcal{G}}\int_{\gamma} v_{s} \eta_s  + 0.
\end{split}
\]
Therefore, formally we derive that \(v\) satisfies (\ref{distributionalintegralequation}).

\section{Derivation of the Ultimate Effective Model}
From now on, we assume that \(\mathcal{G}\) is regular on the flat torus \(\mathbb{T}^2\), and all \(\mathcal{G}\), \(\Gamma_1\), \(R_\delta\) and \(R_\delta^+\) are considered as distributed periodically on \(\R^2\). Notice that the definition of \(\Gamma_1\), \(R_\delta\) and \(R_\delta^+\) are dependent of the choice of coordinate system on the plane.

\subsection{Method I: a traditional treatment}
We start with the traditional method I. It is clear that for each \(\delta\), \(A_\delta^\epsilon(\x)=\bar{\sigma}(\epsilon^{-1}\x)I_{2\times2}\) is uniformly elliptic with lower bound 1 and upper bound \(\sigma=O(1/\delta)\). Therefore, we may directly apply classical results, and obtain a model \(u^{\delta,0}\) as \(\epsilon\rta0\), which satisfies
\be
\bca
-\gd\cdot(\Sigma_\delta\gd u^{\delta,0})=f,&\\
u^{\delta,0}\in H_0^1(\co),&
\eca
\ee
where \(\Sigma_\delta\) is a constant positive definite matrix. Furthermore, we may write
\be
(\Sigma_\delta)_{k,l}=\int_{\msquare} \bar{\sigma}(\x) (\gd w_k+\vec{e}_k)\cdot \vec{e}_l,\,k,l=1,2,
\ee
with \(w_k\) satisfying
\be
\bca
\gd\cdot(\bar{\sigma}(\gd w_k+\vec{e}_k))=0, &\\
w_k\text{ is 1-periodic and }\int_{\msquare}w_k=0.&
\eca\label{correctorequation}
\ee
Here we will call \(w_k\)'s \emph{correctors}, and (\ref{correctorequation}) will be called \emph{corrector equations}. The proper function space \(H_{per}^1(\msquare)\) for the above equation should be the natural completion of 1-periodic \(C_b^\infty(\R^2)\) functions under the norm of \(H^1(\msquare)\). Notice that this function space is different from \(H^1(\msquare)\).

With this function space, by classical theories, we define the \emph{weak solution} to (\ref{correctorequation}) to be the element \(w_k\in H_{per}^1(\msquare)\) satisfying
\be
\int_{\msquare} \bar{\sigma}\gd v\cdot (\gd w_k +\vec{e}_k)=0,\label{weaksolutiontoeffectivecorrector}
\ee
for every \(v\in H_{per}^1(\msquare)\) and \(\int_{\msquare} w_k =0\). According to Riesz representation theorem, we have the existence and uniqueness of a weak solution \(w_k\).

To understand the convergence of \(u^{\delta,0}\) as \(\delta\rta0\), it is needed to establish uniform ellipticity of the tensor family \(\{A_\delta\}_{\delta>0}\). The latter one is determined by the correctors, and thus analysing the corrector equations becomes a crucial problem that opens our further analysis.

\begin{lem}
For \(k=1,2\), we have
\be
\int_{\msquare} \bar{\sigma} |\gd w_k|^2\le O(1),
\ee
with the bound independent of small \(\delta>0\).\label{h1estimatefordeltacorrector}
\end{lem}

\begin{proof}
According to (\ref{weaksolutiontoeffectivecorrector}), we have by setting \(v=w_k\),
\[
\int_{\msquare} \bar{\sigma} |\gd w_k|^2 = \int_{\msquare} \bar{\sigma} \gd w_k \cdot\vec{e}_k.
\]
Using the inequality \(|ab|\le \epsilon a^2+ \frac{b^2}{4\epsilon} \) for every \(\epsilon,a,b>0\), we know that
\[
\int_{\msquare} \bar{\sigma} |\gd w_k|^2 \le 8 \int_{\msquare} \bar{\sigma} = O(1).
\]
\end{proof}

Now, we are ready to discuss the limit of \(w_k\). After passage to a subsequence, we know that \(w_k\) will weakly converge to some \(\hat{w}_k\) in \(H_{per}^1(\msquare)\) as \(\delta\rta0\). To see the equation that \(\hat{w}_k\) satisfies, we focus on one representative cell \(\msquare=(0,1)^2+\mathbf{\lambda}\) with \(\mathbf{\lambda}\in\R^2\), and define
\[
\Tilde{w}_k=w_k+x_k,\,k=1,2.
\]
Clearly, by Lemma (\ref{h1estimatefordeltacorrector}), \(\Tilde{w}_k\) is also uniformly bounded in \(H^1(\msquare)\), and so as \(\delta\rta0\), after passing to a subsequence, it will converge to some \(\bar{w}_k=\hat{w}_k+x_k\). Moreover, it satisfies for every \(\x\in \msquare\),
\[
\gd\cdot(\bar{\sigma}\gd \Tilde{w}_k)(\x)=0.
\]
Writing the above equation in the sense of weak solutions, we have
\[
\int_{\msquare}\bar{\sigma} \gd \eta\cdot \gd \Tilde{w}_k =0, 
\]
for every \(\eta\in C_0^\infty(\msquare)\). By previous work, we know that as \(\delta\rta0\), the above equation becomes
\[
\int_{\msquare} \gd \eta\cdot\gd \bar{w}_k + a \int_{\Gamma_1} \eta_s (\bar{w}_k)_s =0.
\]
Varying \(\mathbf{\lambda}\), this equation forces \(\bar{w}_k\) to satisfy the following integral equation
\be
\int_{\msquare} \gd v\cdot\gd \bar{w}_k + a \int_{\Gamma_1} v_s (\bar{w}_k)_s =0,\label{3.6}
\ee
where \(v\) is any 1-periodic \(C_b^1(\R^2)\) function, and thus \(\hat{w}_k\) must satisfy
\be
\int_{\msquare} \gd v\cdot\gd \hat{w}_k + a \int_{\Gamma_1} v_s (\hat{w}_k)_s+ a \int_{\Gamma_1} v_s (x_k)_s=0.\label{effectivecorrectorintegalequation}
\ee
To understand the above equality, a proper function space \(Z_{\Gamma_1}^{1,per}(\msquare)\) for \(\hat{w}_k\) should be defined:
\[
Z_{\Gamma_1}^{1,per}(\msquare)\coloneqq \lma v\in H_{per}^1(\msquare);\;\restr{v}{\gamma}\in H^1(\gamma),\forall \gamma\in\mathcal{G}\text{ and }\int_{\msquare}v = 0\rma,
\]
on which we introduce a new inner product
\[
(u,v)_{Z_{\Gamma_1}^{1,per}(\msquare)}\coloneqq \int_{\msquare} \gd u\cdot \gd v + a \int_{\Gamma_1} u_s v_s .
\]
After slight modifications of the density result in section (2.2), we see that \(C_b^1(\R^2)\hookrightarrow Z_{\Gamma_1}^{1,per}(\msquare)\) incarnates a dense subset, which ensures that equation (\ref{effectivecorrectorintegalequation}) holds also for \(v\in Z_{\Gamma_1}^{1,per}(\msquare)\). The proper PDE interpretation of the integral equation has already been descibed in Theorem (\ref{1.1}).

By the PDE of \(\hat{w}_k\), it is immediate to see that \(v=0\) is not a solution to the above equation once there is a nonlinear arc in \(\mathcal{G}\), because on this arc the term \((x_k)_{ss}\) is not 0. Moreover, we have uniqueness of weak solutions in \(Z_{\Gamma_1}^{1,per}(\msquare)\) by the integral equation (\ref{effectivecorrectorintegalequation}).

With the convergence of \(w_k\) as \(\delta\rta0\) well-understood, we are able to discuss the convergence of effective tensors \(\Sigma_\delta\). Recalling the formula (\(k,l=1,2\))
\[
\begin{split}
    (\Sigma_\delta)_{k,l}&=\int_{\msquare} \bar{\sigma}(\x) (\gd w_k+\vec{e}_k)\cdot \vec{e}_l d\x\\
&=\int_{\msquare} \bar{\sigma}(\x) \gd \Tilde{w}_k(\x)\cdot\vec{e}_l d\x\\
&=\int_{\msquare} (\Tilde{w}_k)_{x_l}(\x) d\x + (\sigma-1)\int_{R_\delta}\gd\Tilde{w}_k(\x)\cdot\vec{e}_l d\x\\
&=1 + Q,
\end{split}
\]
we then only have to discuss the convergence of \(Q\). Observe that
\[
\begin{split}
    Q&\sim \sigma\int_{R_\delta} \gd \Tilde{w}_k \cdot\vec{e}_l d\x\\
     &=\sigma \int_{R_\delta} \lb \frac{(\Tilde{w}_k)_s}{1+\tau\kappa(s)} \vec{T} + (\Tilde{w}_k)_\tau \vec{n} \rb\cdot \vec{e}_l d\x\\
     &=\sigma\int_0^l\int_0^\delta (\Tilde{w}_k)_s(\vec{T}\cdot \vec{e}_l)+  (\Tilde{w}_k)_\tau(\vec{n}\cdot \vec{e}_l )(1+\tau \kappa(s))d\tau ds\\
     &= \sigma \int_0^l\int_0^\delta (\Tilde{w}_k)_s(s,0^+)(\vec{T}\cdot \vec{e}_l) d\tau ds+o(1)\\
     &\overset{\delta\rta0}{\longrightarrow} a \int_{\Gamma_1} (\bar{w}_k)_s(\vec{T}\cdot \vec{e}_l) ds.
\end{split}
\]
Here \(\vec{T}\) and \(\vec{n}\) are the unit tangent vector field and unit outer normal field induced by \(\Gamma_1\) defined in \(R_\delta\) for small \(\delta>0\), and we have used Lemma (\ref{D2estimate}), transmission conditions (\ref{transmissioncondition}) and trace theorem. Therefore, we arrive at the limit \(\Sigma_0\) of \(\Sigma_\delta\) as \(\delta\rta0\), which is
\be
\begin{split}
    (\Sigma_0)_{kl}&=\delta_{kl}+a \int_{\Gamma_1} (\bar{w}_k)_s(\vec{T}\cdot \vec{e}_l) ds\\
               &=\delta_{kl}+a \int_{\Gamma_1} (\bar{w}_k)_s(x_l)_s ds.
\end{split}
\ee

\begin{lem}
\(\Sigma_0\) is always positive definite.
\end{lem}

\begin{proof}
We have
\be
\begin{split}
    (\Sigma_0)_{kl}&=\delta_{kl}+a \int_{\Gamma_1} (\bar{w}_k)_s(x_l)_s ds\\
                   &=\delta_{kl}+a \int_{\Gamma_1} (\bar{w}_k)_s(\bar{w}_l)_s ds-a\int_{\Gamma_1} (\bar{w}_k)_s(\hat{w}_l)_s ds.
\end{split}\label{calculatesigma}
\ee
Replacing \(v=\hat{w}_l\) in equation (\ref{3.6}), we obtain
\[
\int_{\msquare} \gd\hat{w}_l \cdot\gd \bar{w}_k + a  \int_{\Gamma_1} (\bar{w}_k)_s(\hat{w}_l)_s =0.
\]
Inserting this into (\ref{calculatesigma}), we see
\[
\begin{split}
    (\Sigma_0)_{kl}&=\delta_{kl}-\int_{\msquare} (\bar{w}_k)_{x_l} +\int_{\msquare}\gd \bar{w}_k\cdot\gd \bar{w}_l +a \int_{\Gamma_1} (\bar{w}_k)_s(\bar{w}_l)_s\\
                   &=\int_{\msquare}\gd \bar{w}_k\cdot\gd \bar{w}_l +a \int_{\Gamma_1} (\bar{w}_k)_s(\bar{w}_l)_s.
\end{split}
\]
For \(u,v\in H_{\Gamma_1}^1(\msquare)\) (this space is quite akin to \(Z_{\Gamma_1}^{1,per}(\msquare)\), which is discussed before), we similarly define
\[
((u,v))=\int_{\msquare}\gd u\cdot\gd v +a \int_{\Gamma_1} u_s v_s,
\]
and clearly \(((\cdot,\cdot))\) defines a semi-inner-product on \(H_{\Gamma_1}^1(\msquare)\), which induces a semi-norm \(p(u)=\sqrt{((u,u))}\le \norm{u}_{H_{\Gamma_1}^1(\msquare)}\). It is clear that \(tr(\Sigma_0)>0\), and so to show that \(\Sigma_0\) is positive definite, it suffices to show that \(\det(\Sigma_0)>0\). By simple computations, we have
\[
\begin{split}
   \det(\Sigma_0)&=((\bar{w}_1,\bar{w}_1))((\bar{w}_2,\bar{w}_2))-((\bar{w}_1,\bar{w}_2))^2\\
              &\ge 0,
\end{split}
\]
with equality holds if and only if for some \(\lambda\in\R\), \(p(\bar{w}_1+\lambda\bar{w}_2)=0\), which forces \(\bar{w}_1+\lambda\bar{w}_2=\) some constant \(C\). If the equality holds, then we obtain
\[
\hat{w}_1=-\lambda\hat{w}_2 - \lambda x_2-x_1 +C.
\]
According to periodicity of \(\hat{w}_1\), we know by the above equation, \(\hat{w}_1\) should be discontinuous across \(x_1=0\), but \(\bar{\sigma}\) is always constant 1 in a neighborhood of some point on this line, and so by classical elliptic equation theories, \(\hat{w}_1\) should be continuous, which causes a contradiction.

\end{proof}

\begin{cor}
There is some \(\mu,\delta_0>0\) such that for all \(0<\delta<\delta_0\), we have
\[
\frac{1}{\mu}Id\prec \Sigma_\delta \prec\mu Id.
\]
\end{cor}

With the tensor family \(\{\Sigma_\delta\}\) uniformly strictly elliptic, we clearly obtain energy estimate
\be
\int_{\co} |\gd u^{\delta,0}|^2\lesssim_{\mu,\co} \int_{\co} f^2. 
\ee
Therefore, we see after passage to a subsequence of \(\delta\rta0\), \(u^{\delta,0}\) converges weakly to some \(u^{0,0}\) in \(H^1(\co)\), and because \(\Sigma_\delta\rta\Sigma_0\), as \(\delta\rta0\), we see \(u^{0,0}\) should satisfy
\be
\bca
-\gd\cdot(\Sigma_0\gd u^{0,0})(\x)=f(\x), &\x\in\co,\\
u^{0,0}(\x)=0,&\x\in\po.
\eca
\ee
This then gives the effective model obtained through method I.

\subsection{Method II: homogenization of EBCs}

 The proof is inspired by Tartar's Energy Method \citep{energtartar}. We consider \(\Gamma_1\) as a pattern on a period as we mentioned in the introduction, and at scale \(\epsilon>0\) and in the \((i,j)^T\)-th cell we denote by \(\Gamma_\epsilon^{i,j}\) the copy of \(\epsilon\Gamma_1\) in it. It was shown in the Preliminaries that for each \(1/\epsilon=N\in\mathbb{N}_+\), after sending \(\delta\rta0\), \(u=(u^{0,\epsilon})\) satisfies (even if there are interior boundaries intersecting the outer boundary)
\be
\int_\co \gd u\cdot\gd\psi d\x+\frac{a}{N}\sum_{i,j=-LN}^{LN-1}\int_{\Gamma_\epsilon^{i,j}} u_s\psi_s ds=\int_\co f\psi d\x,\label{uepsilon}
\ee
for every \(\psi\in C_0^\infty(\co)\). In fact, if the pattern is regular, by the density result, \(\psi\) can be chosen to be \(H_{\cup_{i,j}\Gamma_\epsilon^{i,j}}^{1,0}(\co)\) functions with compact support in \(\co\). Moreover, regardless of the regularity of the pattern, we have an energy estimate
\be
\int_{\co} |\gd u|^2d\x +\frac{a}{N}\sum_{i,j=-LN}^{LN-1}\int_{\Gamma_\epsilon^{i,j}} u_s^2 ds \lesssim_{\co} \int_\co f^2 d\x.
\ee

Let \(\mathbf{H}^1(\co)\coloneqq (H^1(\co))^2\) with ``\((\cdot)^2\)" understood as Banach space product. For \(\vec{g}=(g_1,g_2)\in \mathbf{H}^1(\co)\), we have by trace theorem, on each \(\Gamma_\epsilon^{i,j}\), \(\vec{g}\) has a unique restriction to \( L^2(\Gamma_\epsilon^{i,j})\), and
\[
\int_{\Gamma_\epsilon^{i,j}} |\vec{g}|^2 ds \le C \lb\epsilon\int_{R_{\delta}^{\epsilon,i,j}}|D \vec{g}|^2 d\x +\frac{1}{\epsilon}\int_{R_{\delta}^{\epsilon,i,j}} |\vec{g}|^2 \rb,
\]
where \(C>0\) might be dependent of \(\delta>0\) but is independent of \(\epsilon>0\), and \({R_{\delta}^{\epsilon,i,j}}\) is a copy of \(\epsilon R_{\delta}\) accessory to \(\Gamma_\epsilon^{i,j}\). Summing up the above inequality over all index \((i,j)\) of cells, we obtain
\[
\sum_{i,j=-LN}^{LN-1} \int_{\Gamma_\epsilon^{i,j}} |\vec{g}|^2 ds \lesssim \epsilon \int_{\co} |D \vec{g}|^2d\x +\frac{1}{\epsilon} \int_{\co} |\vec{g}|^2d\x,
\]
and by dividing both sides by \(N\), we see
\be
\frac{1}{N}\sum_{i,j=-LN}^{LN-1} \int_{\Gamma_\epsilon^{i,j}} |\vec{g}|^2 ds \le C \norm{\vec{g}}_{\mathbf{H}^1(\co)}^2,\label{L2boundaryestimate}
\ee
with \(C>0\) independent of \(\epsilon>0\). The above discussions lead to the following consequence.
\begin{lem}
For each \(\vec{g}\in \mathbf{H}^1(\co)\), the following family of linear functionals
\[
U_{N}(\vec{g})\coloneqq \int_\co \gd u\cdot\vec{g} dxdy+\frac{a}{N}\sum_{i,j=-LN}^{LN-1}\int_{\Gamma_\epsilon^{i,j}} u_s\lb\vec{g}\cdot\vec{T}\rb ds
\]
is uniformly bounded in \(\lb\mathbf{H}^1(\co)\rb^\ast\). 
\end{lem}
Therefore, by Banach-Alaoglu Theorem, after passage to a subsequence of \(N\rta\infty\), we know that \(U_N\) \(\ast\)-weakly converges to some \(U_\infty\) in \(\lb\mathbf{H}^1(\co)\rb^\ast\). On the other hand, by energy estimate, we know that \(u^{0,\epsilon}\) is uniformly bounded in \(H_0^1(\co)\), and thus after passage to a subsequence of \(\epsilon=1/N\rta0\), \(u^{0,\epsilon}\) will weakly converge to some \(u^\ast\) in \(H_0^1(\co)\). It is natural to ask whether \(u^\ast\) and \(u^{0,0}\) are the same. To this end, one needs to find out the integral representation of \(U_\infty\) in terms of \(u^\ast\).

To derive the true effective model, one needs to be aware that, although the road width \(\delta>0\) already sent to 0 and heterogeneity replaced by boundary conditions, the process of homogenization continues to be the same. Looking back to Method I, we find that the effective corrector \(\hat{w}_k\) should be useful. We then follow the way of classical homogenization, and define an auxiliary function
\[
w_k^\epsilon(\x)\coloneqq x_k + \epsilon \hat{w}_k(\epsilon^{-1}\x).
\]
By the equation (\ref{effectivecorrectorintegalequation}) that \(\hat{w}_k\) satisfies, we have that for all \(\psi\in C_0^\infty(\co)\),
\be
\int_{\co} \gd w_k^\epsilon\cdot \gd \psi d\x + \frac{a}{N}\sum_{i,j=-LN}^{LN-1}\int_{\Gamma_\epsilon^{i,j}} (w_k^\epsilon)_s\psi_s ds=0.\label{wepsilon}
\ee
Replacing \(\psi\) by \(\psi w_k^\epsilon\) in (\ref{uepsilon}) and by \(\psi u^{0,\epsilon}\) (\(u=u^{0,\epsilon}\) if the symbols are clearly understood) in (\ref{wepsilon}), and combining the two equations, we obtain
\be
\int_{\co} \gd u\cdot\gd\psi w_k^\epsilon d\x -\int_{\co} \gd w_k^\epsilon\cdot\gd\psi u d\x + \frac{a}{N}\sum_{i,j=-LN}^{LN-1}\lmb\int_{\Gamma_\epsilon^{i,j}} u_s\psi_s w_k^\epsilon ds -\int_{\Gamma_\epsilon^{i,j}} \lb w_k^\epsilon\rb_s\psi_s u ds\rmb=\int_\co f w_k^\epsilon \psi d\x.
\ee
With the notations defined before, we have
\[
\begin{split}
    LHS&= U_N(\gd\psi x_k)+ \epsilon U_N(\gd\psi \hat{w}_k(\epsilon^{-1}\cdot))-\int_{\co} \gd w_k^\epsilon\cdot\gd\psi u d\x-\frac{a}{N}\sum_{i,j=-LN}^{LN-1} \int_{\Gamma_\epsilon^{i,j}} \lb w_k^\epsilon\rb_s\psi_s u ds\\
      &\eqqcolon U_N(\gd\psi x_k)+I-II-III.
\end{split}
\]
Recalling that \(U_N\overset{\ast}{\weakcv}U_\infty\) in \((\mathbf{H}^1(\co))^\ast\), \(u\) converges strongly to \(u^\ast\) in \(L^2(\co)\), and \(\gd w_k^\epsilon\) converges weakly to \(\vec{e}_k\) in \(L^2(\co)\) as \(N\rta\infty\), we immediately obtain that as \(\epsilon\rta0\),
\begin{itemize}
    \item \(U_N(\gd\psi x_k) \longrightarrow U_\infty(\gd \psi x_k)\);
    \item \(II \longrightarrow \int_{\co} u^\ast\psi_{x_k}\).
\end{itemize}
Moreover, by uniform boundedness of \(U_N\), we have
\[
\begin{split}
   I&\le \epsilon \norm{U_N}_{(\mathbf{H}^1(\co))^\ast}\norm{ \gd\psi}_{L^\infty(\co)} \norm{\hat{w}_k}_{Z_{\Gamma_1}^{1,per}(\msquare)}\\
 &\lesssim \epsilon  \norm{ \psi}_{C^1(\co)} \norm{\hat{w}_k}_{Z_{\Gamma_1}^{1,per}(\msquare)}.
\end{split}
\]
It then suffices to consider \(III\), but we have
\be
\begin{split}
   III&=\frac{a}{N}\sum_{i,j=-LN}^{LN-1} \int_{\Gamma_\epsilon^{i,j}} \lb w_k^\epsilon\rb_s\psi_s u ds\\
      &=\frac{a}{N}\sum_{i,j=-LN}^{LN-1} \int_{\Gamma_\epsilon^{i,j}} \lb  w_k^\epsilon\rb_s \vec{T}\cdot \gd\psi u ds.
\end{split}\label{3sum}
\ee
To find the limit of the above quantity, let us start with simple patterns. Suppose that a 1-periodic pattern \(\mathcal{G}\) is composed of only one arc \(\gamma\) of length \(l>0\) in each cell. To specify the scaling relationship, we assume that the arc in \([0,1]^2\) is reparametrized by \(\gamma_1^{0,0}:[0,l]\rta[0,1]^2\) by arc-length. If a copy is in \([i,i+1]\times[j,j+1]\), then it can also be reparametrized by arc-length by \(\gamma_1^{0,0}+(i,j)^T\). At scale \(\epsilon>0\), the copy \(\epsilon\gamma\) in the \((i,j)^T\)-th cell can still be reparametrized by \(\gamma_{\epsilon}^{i,j}(\cdot)=\epsilon \gamma_{1}^{i,j}(\epsilon^{-1}\cdot)\) by arc-length. Notice that
\[
w_k^\epsilon(\gamma_{\epsilon}^{i,j}(s))=\epsilon\bar{w}_k(\epsilon^{-1}\gamma_{\epsilon}^{i,j}(s))=\epsilon \bar{w}_k(\gamma_1^{0,0}(\epsilon^{-1}s)),\,s\in[0,\epsilon l],
\]
and thus we have
\[
\lb w_k^\epsilon(\gamma_{\epsilon}^{i,j}(\cdot)) \rb_s(s) = \lb\bar{w}_k(\gamma_1^{0,0}(\cdot))\rb_s(\epsilon^{-1}s).
\]
Moreover, on \(\Gamma_\epsilon^{i,j}\), we have the unit tangent field \(\vec{T}_{(\epsilon,i,j)}(s)=\vec{T}_{(1,0,0)}(\epsilon^{-1}s),\,s\in[0,\epsilon l]\). Therefore, by the above discussions, the sum (\ref{3sum}) can be formally understood as
\[
\frac{1}{N} \sum_{i,j=-LN}^{LN-1}\int_{\Gamma_{\epsilon}^{i,j}} \vec{F}(\epsilon^{-1}s)\cdot \vec{G}_\epsilon(\gamma_\epsilon^{i,j}(s)) ds,
\]
where \(\vec{F}\) is \((\bar{w}_k)_s \vec{T}\), and \(\vec{G}_\epsilon\) is a weakly convergent sequence in \(\mathbf{H}^1(\co)\) with support of its elements contained in a compact subset of \(\co\). The case for a general pattern is simply a sum like this over all arcs. The following lemmas focus on the simple case that \(\mathcal{G}\) is composed of only one arc, and vector-valued functions replaced by scalar-valued ones.

\begin{lem}
Let \(G_\epsilon\in H_0^1(\co)\) be weakly convergent to \(G_0\) in \(H_0^1(\co)\). Then, we have
\[
\frac{1}{N} \sum_{i,j=-LN}^{LN-1}\int_{\Gamma_{\epsilon}^{i,j}} {G}_\epsilon(\gamma_\epsilon^{i,j}(s)) d s\overset{N\rta\infty}{\longrightarrow} l \int_{\co} G_0(\x) d\x.
\]
\end{lem}

\begin{proof}
Let \(K\in\mathbb{N}_+\) large, and we consider the uniform partition of \([0,l]\) by mesh \(l/K\). Let \(q_k=\gamma_1^{0,0}(l k/K)\) for \(k=0,\cdots,K\). When \(K\) is large, then because \(\gamma\) is \(C^2\), each arc \(\bigfrown{q_kq_{k+1}}\) can be written as the graph of function \(g_k(\bar{s})\) on the linear segment \(\overline{q_kq_{k+1}}\), where \(\bar{s}\in\overline{q_kq_{k+1}}\). 

According to Lemma (\ref{functiontranslate}), we know that enduring a small error \(o_K(1)\), the integral of \(G_\epsilon\) on the arc \(\bigfrown{q_kq_{k+1}}\) can be replaced by that on the linear segment \(\overline{q_kq_{k+1}}\). After proper translation (Lemma (\ref{translationerror})), rotation (Lemma (\ref{rotationfact})) and rescaling (Lemma (\ref{rescalefact})), we only have to consider 
\[
\frac{\sum_{k=0}^{K-1}|q_k-q_{k+1}|}{N}\sum_{j=-LN}^{LN-1} \int_{\R} G_\epsilon(x, j/N) dx.
\]
But \(\int_{\R} G_\epsilon(x, y) dx\)'s are uniformly \(C_0^{1/2}(\R)\) functions in \(y\), and so we arrive at
\[
\sum_{k=0}^{K-1}|q_k-q_{k+1}| \int_{\co} G_\epsilon(\x) d\x,
\]
which converges to 
\[
\sum_{k=0}^{K-1}|q_k-q_{k+1}| \int_{\co} G_0(\x) d\x,
\]
by strong convergence of \(G_\epsilon\) in \(L^2(\co)\) as \(\epsilon\rta0\). Because \(\gamma\) is rectifiable, we know that as \(K\rta\infty\), 
\[
\sum_{k=0}^{K-1}|q_k-q_{k+1}| \longrightarrow l.
\]

\end{proof}

\begin{cor}
Let \(\phi\) be a step function on \(\gamma\), then we have
\[
\frac{1}{N} \sum_{i,j=-LN}^{LN-1}\int_{\Gamma_{\epsilon}^{i,j}} \phi(\epsilon^{-1}s){G}_\epsilon(\gamma_\epsilon^{i,j}(s)) ds\overset{N\rta\infty}{\longrightarrow} \int_{\gamma} \phi(s) ds\cdot \int_\co G_0(\x) d\x.
\]\label{simplefunctioncase}
\end{cor}

\begin{lem}
We have the following estimate
\[
\lw\frac{1}{N} \sum_{i,j=-LN}^{LN-1}\int_{\Gamma_{\epsilon}^{i,j}} \phi(\epsilon^{-1}s){G}_\epsilon(\gamma_\epsilon^{i,j}(s)) ds\rw \le CL \norm{G_\epsilon}_{H_0^1(\co)} \norm{\phi}_{L^2(\gamma)}.
\]
\end{lem}

\begin{proof}
By Cauchy-Schwartz, we have
\[
\begin{split}
    \lw\frac{1}{N} \sum_{i,j=-LN}^{LN-1}\int_{\Gamma_{\epsilon}^{i,j}} \phi(\epsilon^{-1}s){G}_\epsilon ds\rw&\le \lb\frac{1}{N} \sum_{i,j=-LN}^{LN-1}\int_{\Gamma_{\epsilon}^{i,j}} {G}_\epsilon^2 ds\rb^{1/2}\lb\frac{1}{N} \sum_{i,j=-LN}^{LN-1}\int_{\Gamma_{\epsilon}^{i,j}} \phi^2(\epsilon^{-1}s) ds\rb^{1/2}\\
    &\overset{(\ref{L2boundaryestimate})}{\le} C\norm{G_\epsilon}_{H_0^1(\co)} \lb\frac{\epsilon}{N} \times (LN)^2 \int_\gamma \phi^2(s) ds \rb^{1/2}\\
    &\le C L \norm{G_\epsilon}_{H_0^1(\co)} \cdot \lb\int_{\gamma} \phi^2(s)ds\rb^{1/2}.
\end{split}
\]
\end{proof}

\begin{cor}
Corollary (\ref{simplefunctioncase}) can be extended to the case that \(\phi\in L^2(\gamma)\). Moreover, because \((\bar{w}_k)_s\vec{T}\) is clearly \(L^2(\gamma)\), the convergence of \(III\) is assured.
\end{cor}

With these lemmas at hand, we obtain
\[
\begin{split}
   III&=\frac{a}{N}\sum_{i,j=-LN}^{LN-1} \int_{\Gamma_\epsilon^{i,j}} \vec{F}(\epsilon^{-1}\cdot)\cdot \vec{G} ds\\
   &\overset{\epsilon\rta0}{\longrightarrow}a\int_{\Gamma_1}\vec{F} ds \cdot \int_\co \vec{G}_0 d\x\\
   &=a\int_{\Gamma_1}(\bar{w}_k)_s\vec{T}ds \cdot \int_\co \gd\psi u^\ast d\x\\
   &=a\int_{\Gamma_1}(\bar{w}_k)_s(x_1,x_2)_s^Tds \cdot \int_\co \gd\psi u^\ast d\x.
\end{split}
\]
Collecting all the terms after sending \(\epsilon\rta0\), we obtain an equation
\be
U_\infty(\gd\psi x_k)-\int_\co u^\ast\psi_{x_k}-a\int_{\Gamma_1}(\bar{w}_k)_s(x_1,x_2)_s^Tds \cdot \int_\co \gd\psi u^\ast d\x=\int_{\co} f \psi x_k d\x.\label{aftermath}
\ee
Recalling the integral equation (\ref{uepsilon}), we know that 
\[
LHS=U_\infty\lb \gd (\psi x_k)\rb= U_\infty(\gd\psi x_k)+U_\infty(\vec{e}_k \psi),
\]
which, combining with (\ref{aftermath}), provide equality
\[
U_\infty(\vec{e}_k \psi)=\int_\co (u^\ast)_{x_k}\psi+a\int_{\Gamma_1}(\bar{w}_k)_s(x_1,x_2)_s^Tds \cdot \int_\co \psi\gd u^\ast d\x=0,\,k=1,2.
\]
Replacing \(\psi\) by \(\eta_{x_1}\) if \(k=1\) and \(\eta_{x_2}\) if \(k=2\) for some \(\eta\in C_0^\infty(\co)\), and summing up over \(k=1,2\), we see
\[
\int_\co f\eta d\x=U_\infty(\gd\eta)=\int_{\co} \gd u^\ast\cdot(\Sigma_0\gd\eta) d\x.
\]
But the solution is unique, and so \(u^\ast=u^{0,0}\).

\section{Analysis on the Trace of the Effective Diffusion Tensor}

\subsection{Optimization and Balance}

One interesting question at this stage is that, is there a proper way to compare the efficiencies between different patterns? It is very natural to compare two effective diffusion tensors \(\Sigma_0\) and \(\Sigma_0'\) by considering whether 
\[
\Sigma_0\succcurlyeq\Sigma_0'.
\]
However, this straight-forward method is not very practical in use, and so we turn to consider the trace of the effective diffusion tensors \(\mathcal{E}\coloneqq tr\lb\Sigma_0\rb.\) This quantity measures the total effect of the enhancement of the given materials and it is, according to Linear Algebra, certainly a weaker form of comparison by semi-positive-definiteness. 

By classical homogenization \citep{tartarbook}, we have the estimate
\[
\frac{1}{\int_{\msquare}\frac{1}{\bar{\sigma}}} Id \preccurlyeq \Sigma_\delta \preccurlyeq \int_{\msquare} \bar{\sigma} Id,
\]
which, by sending \(\delta\rta0\), gives that
\[
Id\preccurlyeq \Sigma_0 \preccurlyeq (1+a l)Id ,
\]
where \(l\) is the total length of \(\Gamma_1\) in one cell at scale 1. The above inequality give rises to 
\[
2\le \mathcal{E} \le 2+2{a l},
\]
and thereby we ask that when could the right-hand-side inequality replaced by equality?

\vspace{0.4cm}

In fact, the right hand side of the above inequality can be modified: if we take a look at each component of \(\Sigma_0\), we have for \(k,l=1,2\),
\be
\begin{split}
     (\Sigma_0)_{kl} &=\delta_{kl}+a \int_{\Gamma_1} (\bar{w}_k)_s(x_l)_s ds\\
                     &=\delta_{kl}+a \int_{\Gamma_1} (x_k)_s(x_l)_s ds +a \int_{\Gamma_1} (\hat{w}_k)_s(x_l)_s ds,
\end{split}\label{sigmakl}
\ee
where by equation (\ref{effectivecorrectorintegalequation}),
\be
a \int_{\Gamma_1} (\hat{w}_k)_s(x_l)_s ds = -\int_{\msquare} \gd \hat{w}_k\cdot\gd \hat{w}_{l} -a \int_{\Gamma_1} (\hat{w}_k)_s(\hat{w}_l)_s.\label{yyds}
\ee
Therefore, we have
\[
(\Sigma_0)_{kk}=1 + a \int_{\Gamma_1} (x_k)_s^2 -\int_{\msquare} \lw\gd \hat{w}_k\rw^2 -a \int_{\Gamma_1} (\hat{w}_k)_s^2,
\]
which implies that
\be
\begin{split}
    \mathcal{E}&=2+ a\int_{\Gamma_1}\lb(x_1)_s^2+(x_2)_s^2\rb -\int_{\msquare} \lb\lw\gd \hat{w}_1\rw^2+\lw\gd \hat{w}_2\rw^2\rb -a \int_{\Gamma_1} \lb(\hat{w}_1)_s^2+(\hat{w}_2)_s^2\rb\\
    &=2+a l-(energy(\hat{w}_1)+energy(\hat{w}_2))\\
    &\le 2 + a l.
\end{split}\label{2jiaal}
\ee
This inequality also ushers us to the maximization of trace of effective diffusion tensors: to maximize \(\mathcal{E}\), it suffices to minimize the sum of energies of \(\hat{w}_1\) and \(\hat{w}_2\), which means that we should find patterns that render them zero. Indeed, the \emph{balanced} patterns defined before exactly fulfill this property.

\vspace{0.4cm}

\begin{proof}[proof of Theorem (\ref{1.2})]
By density result, we know that if \(\mathcal{G}\) is regular when considered as a 1-periodic pattern on \(\R^2\), each of the corresponding effective corrector equations for \(k=1,2\) admits a unique weak solution in \(Z_{\Gamma_1}^{1,per}(\msquare)\). Thus, one only have to consider the equivalent conditions that ensures \(\hat{w}_1\equiv\hat{w}_2\equiv0\) to be solutions. By the boundary conditions described in Theorem (\ref{1.1}), we know that the three conditions proposed in Theorem (\ref{1.2}) are evidently both sufficient and necessary.
\end{proof}

\vspace{0.4cm}

\subsection{Several Examples of Balanced Patterns}

The classification of balanced patterns will be a distinct subject, and here we only list some examples. It is certain that the following four cases do not cover all possible ones.

\begin{figure}[htbp]
        \centering
        \includegraphics[width=8cm]{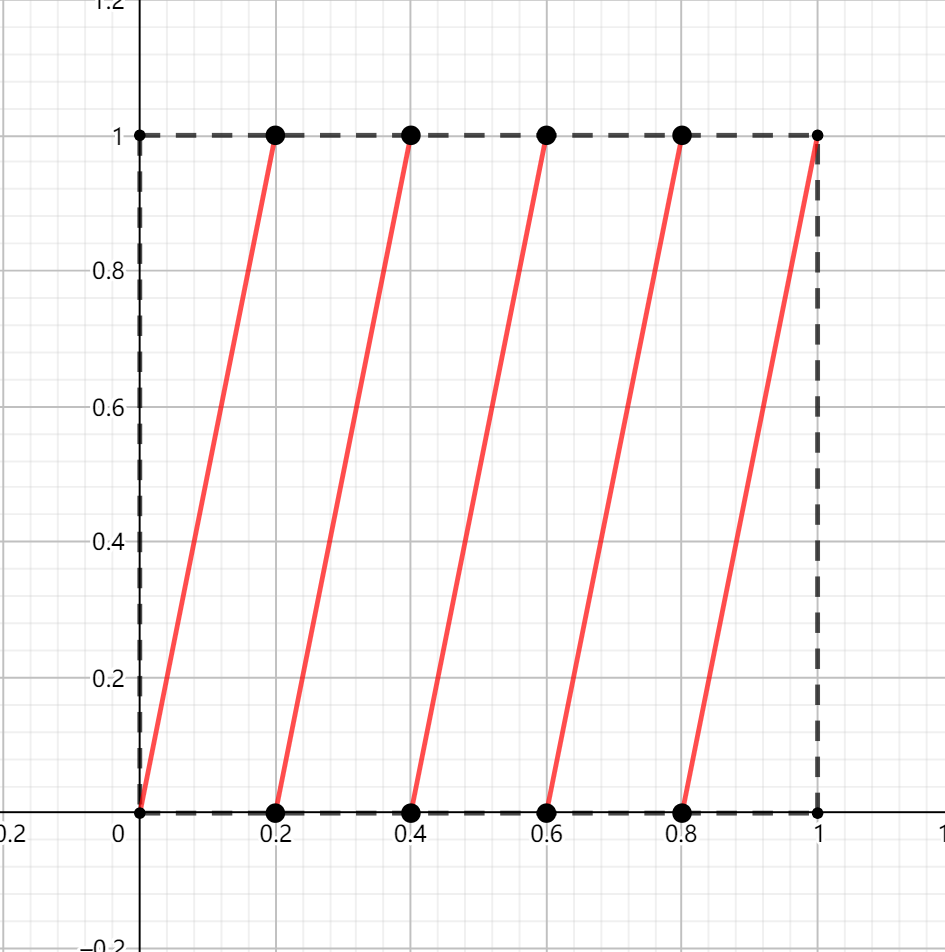}
        \caption{[\textbf{Lines}] Straight lines that take rational slopes are akin to this picture}
         \label{linenoend}
\end{figure}

\begin{figure}[htbp]
     \begin{subfigure}{0.3\textwidth}
         \includegraphics[width=1.6\textwidth]{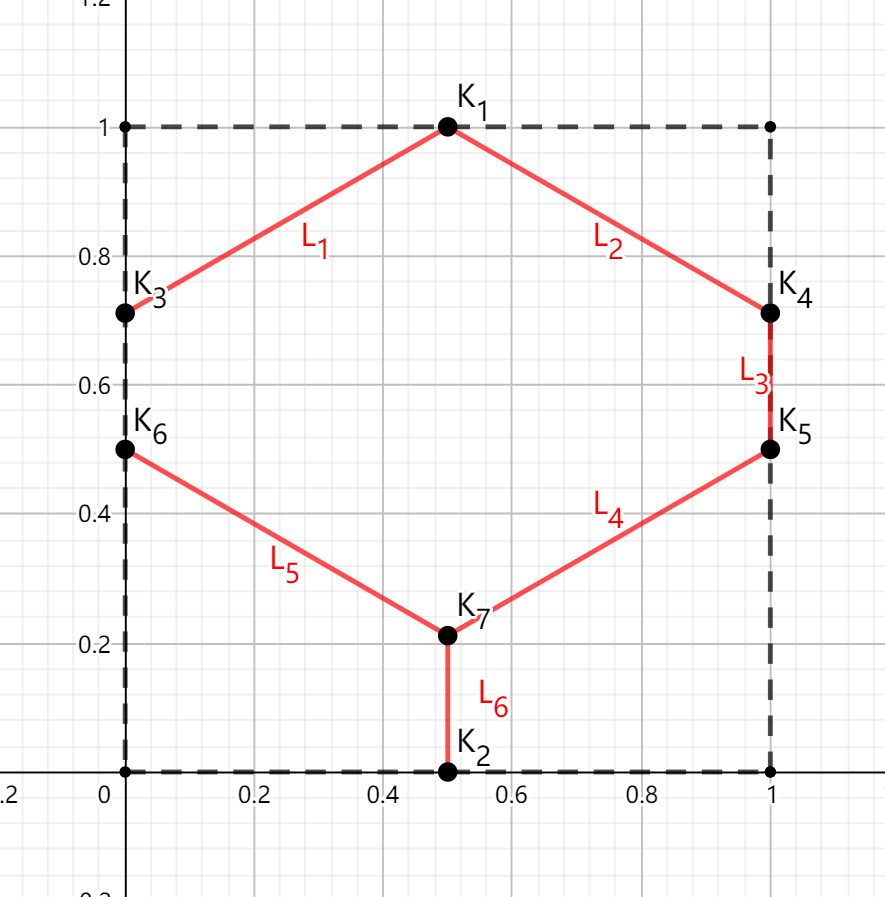}
         \caption{Viewed in one cell.}
         \label{hexagonsa}
     \end{subfigure}
     \hspace{1.2in}
     \begin{subfigure}{0.3\textwidth}
         \includegraphics[width=1.7\textwidth]{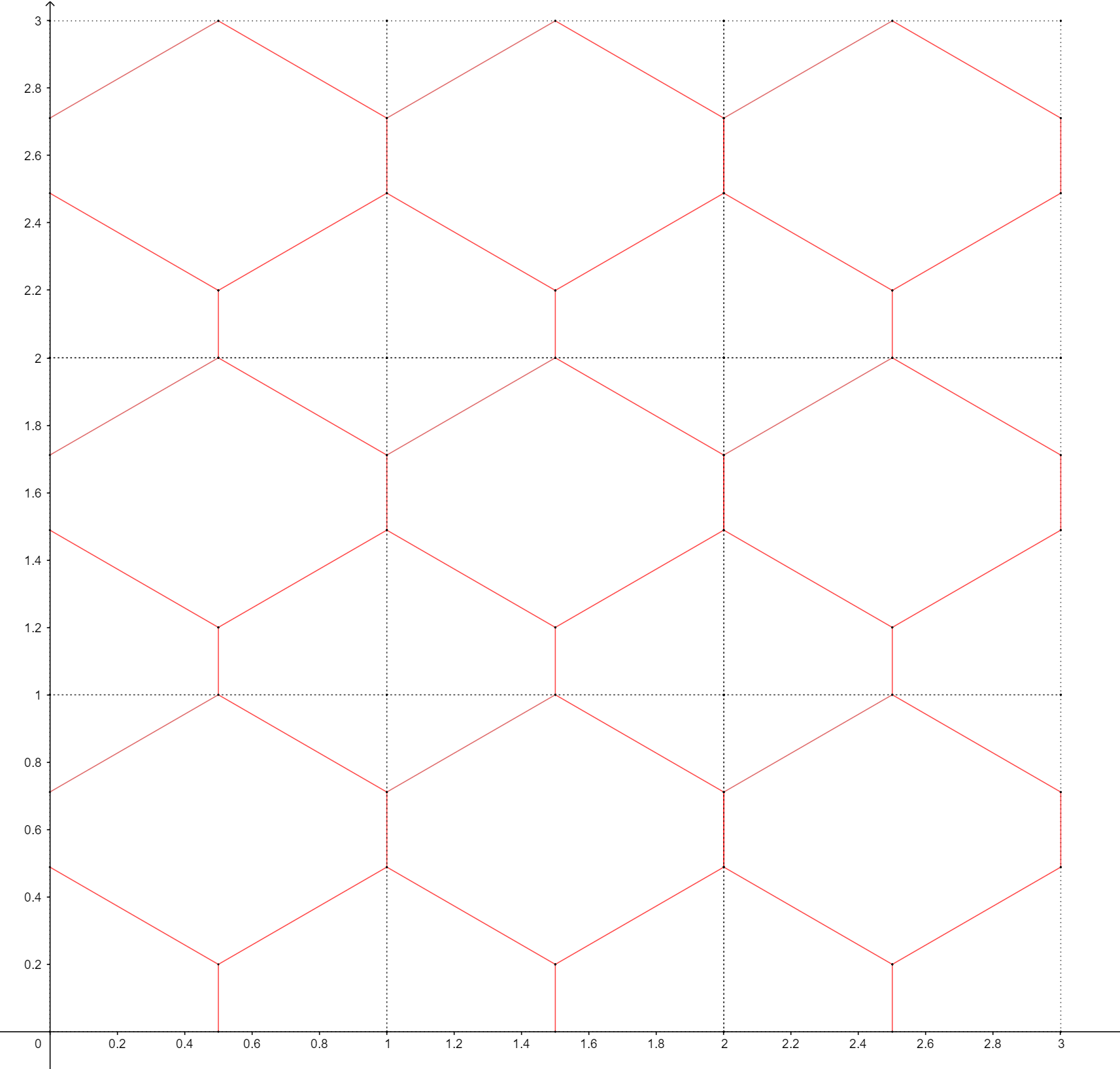}
         \caption{Viewed in a nine-palace.}
         \label{hexagonsb}
     \end{subfigure}
        \caption{[\textbf{Hexagons}] Hexagons with each end point connecting to three arcs can be modified to be balanced}
        \label{hexagons}
\end{figure}

\begin{figure}[htbp]
     \begin{subfigure}{0.3\textwidth}
         \includegraphics[width=1.6\textwidth]{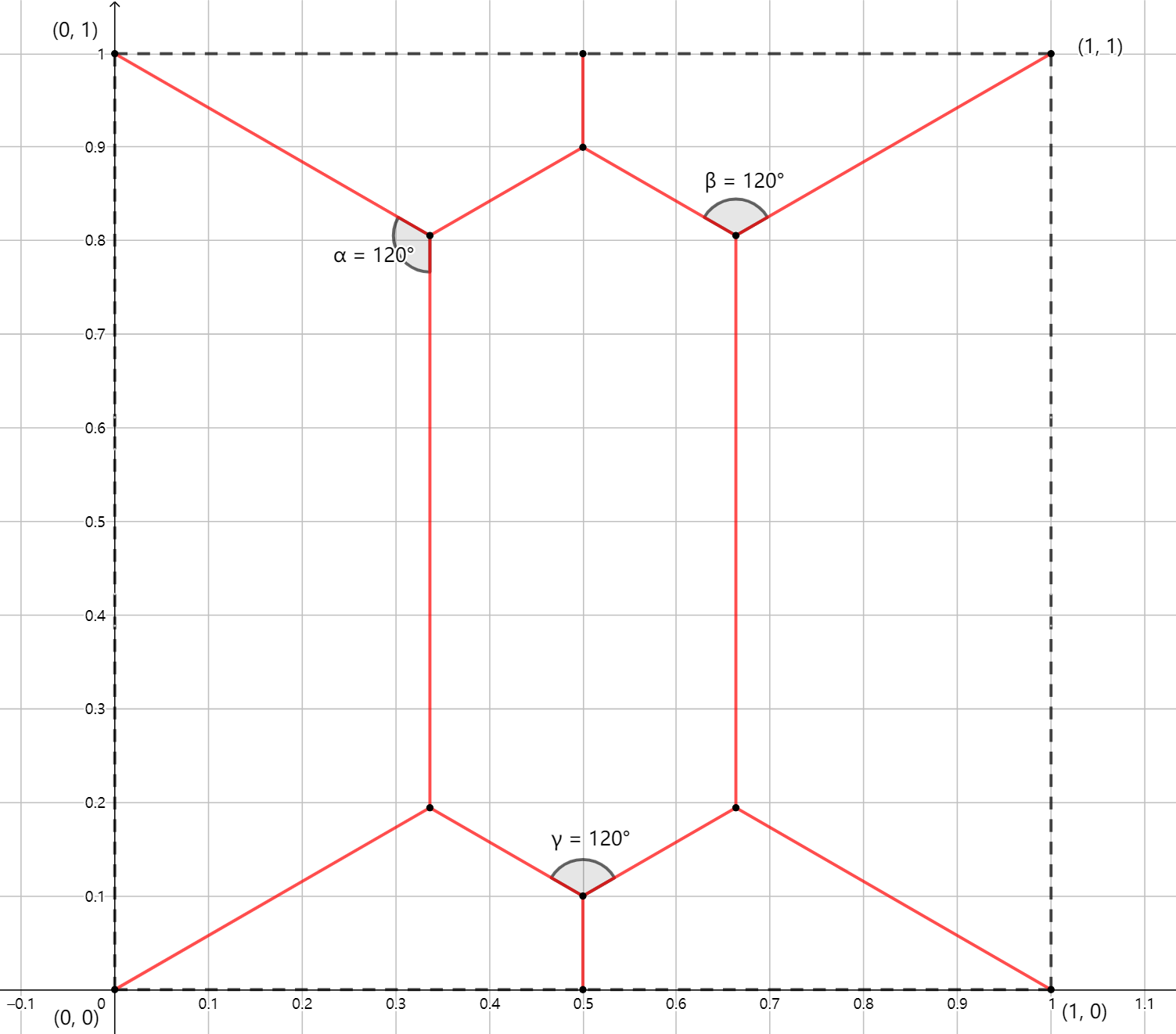}
         \caption{Viewed in one cell.}
         \label{5&hexagonsa}
     \end{subfigure}
     \hspace{1.2in}
     \begin{subfigure}{0.3\textwidth}
         \includegraphics[width=1.7\textwidth]{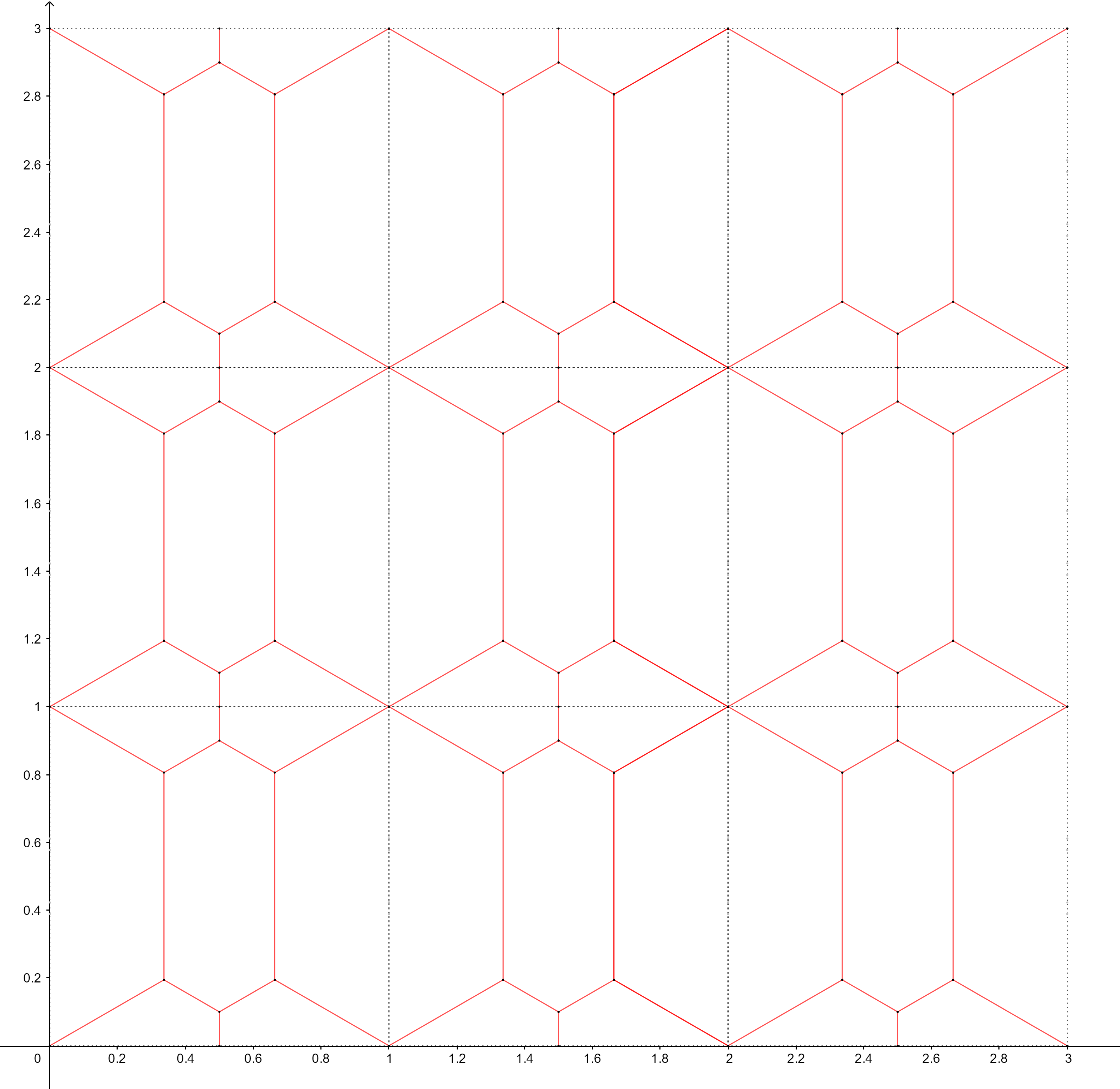}
         \caption{Viewed in a nine-palace.}
         \label{5&hexagonsb}
     \end{subfigure}
        \caption{[\textbf{Others}] Balanced pattern composed of Pentagons and Hexagons}
        \label{5&hexagons}
\end{figure}

\begin{figure}[htbp]
     \begin{subfigure}{0.3\textwidth}
         \includegraphics[width=1.6\textwidth]{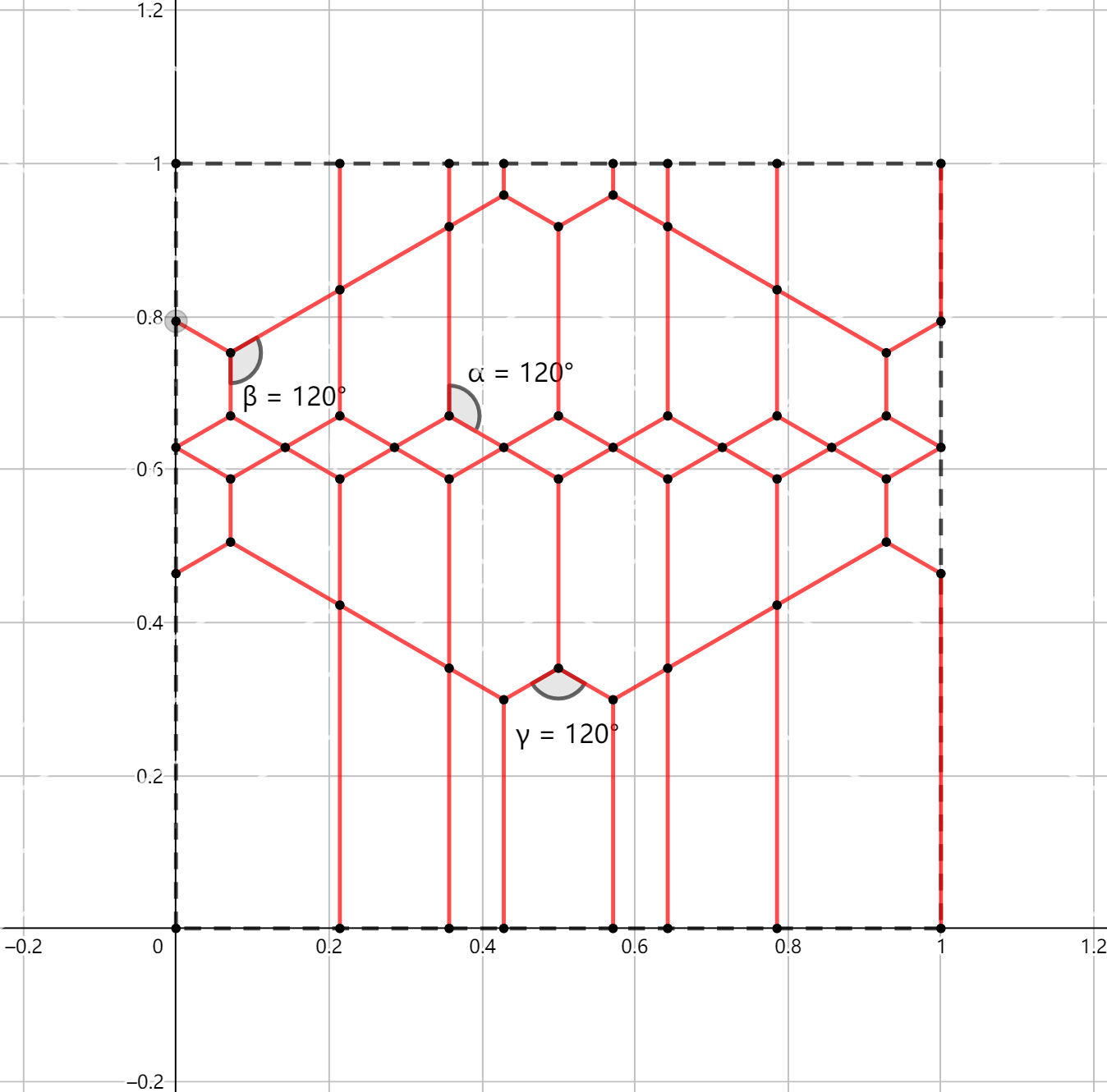}
         \caption{Viewed in one cell.}
         \label{Mixturesa}
     \end{subfigure}
     \hspace{1.2in}
     \begin{subfigure}{0.3\textwidth}
         \includegraphics[width=1.7\textwidth]{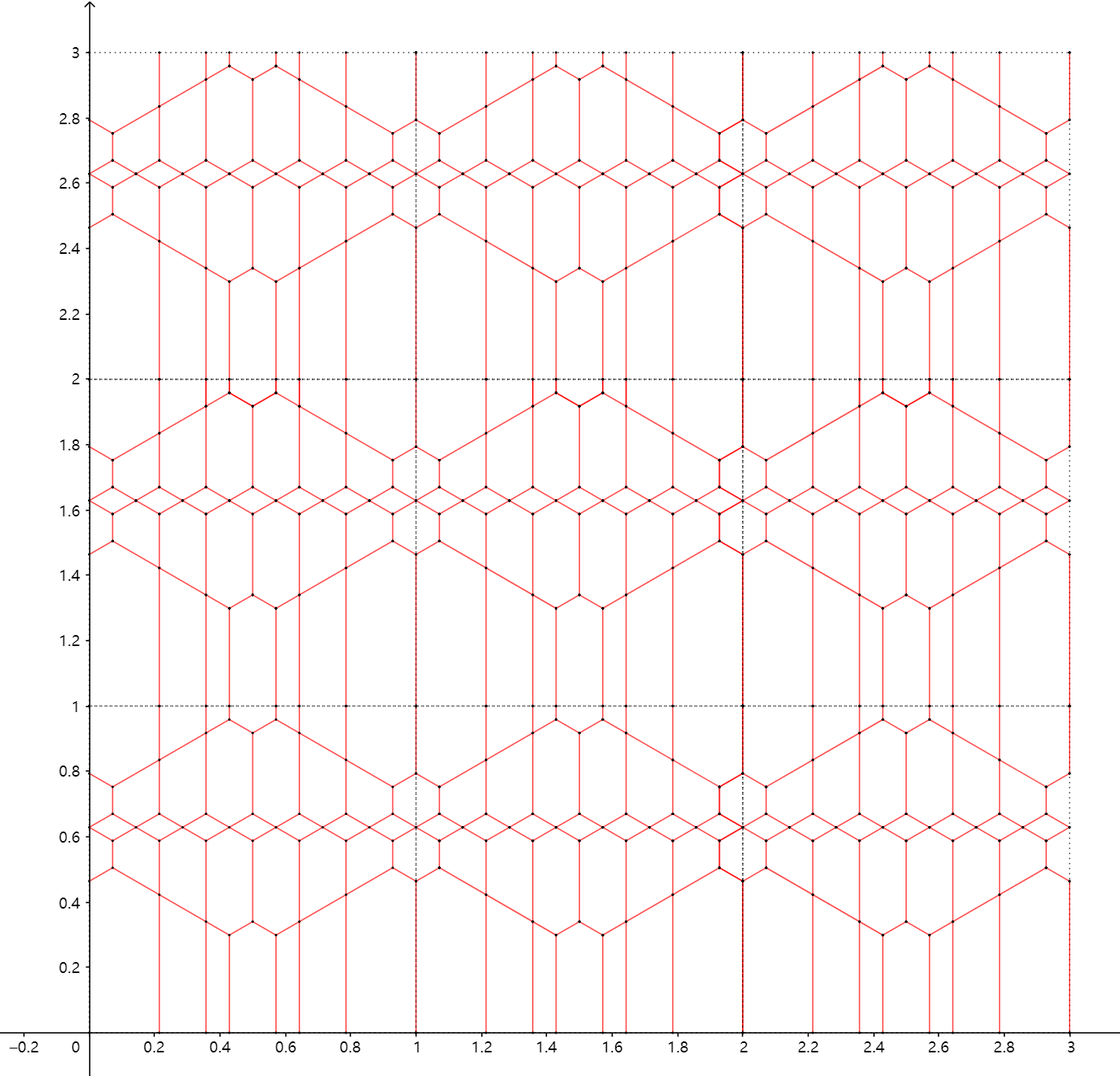}
         \caption{Viewed in a nine-palace.}
         \label{Mixturesb}
     \end{subfigure}
        \caption{[\textbf{Others}] Balanced pattern composed of Parallelograms, Pentagons and Hexagons}
        \label{Mixtures}
\end{figure}

\clearpage

\subsection{As \(a\rta\infty\)}
This subsection mainly discusses the asymptotics of \(tr(\Sigma_0)\) as \(a\rta\infty\). Before everything starts, we briefly prove the asymptotics of the trace as \(a\rta0^+\). According to equation (\ref{sigmakl}), we only have to evaluate the term \(-a\int_{\Gamma_1} (\hat{w}_k)_s(x_k)_s\). By an integration by parts, we obtain
\[
\begin{split}
    -a\int_{\Gamma_1} (\hat{w}_k)_s(x_k)_s&=a\int_{\Gamma_1}\hat{w}_k(x_k)_{ss} + O\lb\sum_{V \text{ is a node of }\mathcal{G}}a |\hat{w}_k(V)|\rb\\
    &\eqqcolon I+II.
\end{split}
\]
Observe that by trace theorem
\[
|I|\le a \lb\int_{\Gamma_1}\hat{w}_k^2\rb^{1/2} \lb\int_{\Gamma_1}(x_k)_{ss}\rb^{1/2}\lesssim a \lb\int_{\msquare}|\gd \hat{w}_k|^2\rb^{1/2},
\]
and similarly
\[
|II|\lesssim a \lb\int_{\msquare}|\gd \hat{w}_k|^2+ a \int_{\Gamma_1}(\hat{w}_k)_s^2\rb^{1/2}.
\]
Combining these two, (\ref{yyds}) and (\ref{2jiaal}), we have 
\[
2+a l-O(a^2)\le tr(\Sigma_0) \le 2+a l.
\]
This estimate shows that when \(a\) is very small, the major contribution of different patterns get close. In order to see the differences among different patterns, we attempt to send \(a\rta\infty\).

\subsubsection{Some Function Spaces}
Given a pattern \(\mathcal{G}\) on a flat torus and a proper periodic extension of it on a plane \(\R^2\), we see that probably a curve \(\gamma\in\mathcal{G}\) crosses the boundary \(\partial \msquare\) of a period \(\msquare=(0,1)^2\). This case is well exhibited in Figure \ref{1typicalexample}, where \(\gamma_1\) and ``\(\gamma_6\)" are considered as one single curve. However, it is clear to see that if one translates the coordinate system of \(\R^2\), \(\gamma_1\) and ``\(\gamma_6\)" combined will be contained in the unit cell (see Figure \ref{translate}). In previous discussions, we simply ignored the influence caused by different translations of coordinate system of \(\R^2\) on the ``real" pattern displayed in the unit cell \(\msquare\), but in this subsection, we will carefully point out the nuances because this is significant in the establishment of the desired estimates.

\begin{figure}[htbp]
        \centering
        \includegraphics[width=6cm]{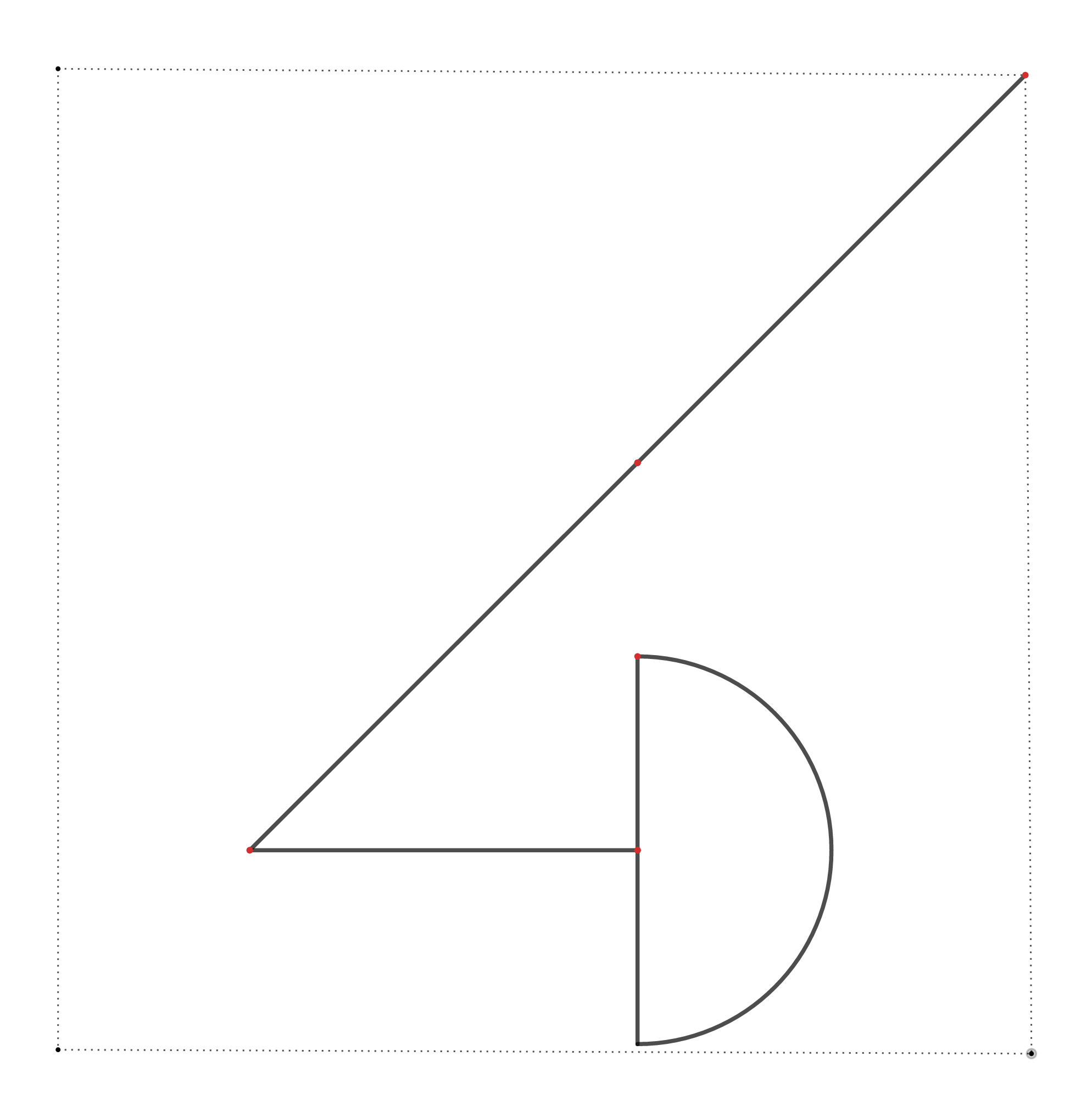}
        \caption{translation of coordinate system in Figure \ref{1typicalexample} makes {\(\gamma_1\)} to avoid crossing the boundary of the unit cell.}
         \label{translate}
\end{figure}

We reuse the original definition of \(\Gamma_1\) in the introduction, where \(\Gamma_1\) is exactly the subset of \(\mathbb{T}^2\) comprising curves in \(\mathcal{G}\). We now differentiate \(\Gamma_1\) and \(\Gamma_1^{\#}\), where the latter is the restriction to one period \(\overline{\msquare}\) of the periodic extension of \(\Gamma_1\) following the extension of \(\mathcal{G}\). Be aware that the latter is dependent on the translation of coordinate system on \(\R^2\), while the prior is not. Collecting the maximal regular components of \(\Gamma_1^{\#}\) in \(\overline{\msquare}\) regardless of the periodicity, we obtain a new pattern \(\mathcal{G}^{\#}\). Observing that in some cases (such as the linear segment \(L_3\) in Figure \ref{hexagonsa}) there are arcs included by the boundary \(\partial \msquare\), we delete one of the arsing two copies in defining \(\Gamma_1^{\#}\).  

\begin{dfn}
We define 
\[
\mathcal{M}^1(\Gamma_1^{\#})\coloneqq \lma u\text{ is measurable on }\Gamma_1^{\#};\;\begin{aligned}
   & \restr{u}{\gamma}\in H^1(\gamma),\,\forall\gamma\in\mathcal{G}^{\#}, \text{ and }\forall V\in \msquare, \,\\
   & \restr{u}{\gamma}(V)=\restr{u}{\gamma'}(V) \text{ for any }\gamma,\gamma'\in\mathcal{G}^{\#}\text{ joint at }V.
\end{aligned}\rma,
\]
and 
\[
\mathcal{M}_{per}^1(\Gamma_1)\coloneqq \lma u\text{ is measurable on }\Gamma_1;\;\begin{aligned}
   & \restr{u}{\gamma}\in H^1(\gamma),\,\forall\gamma\in\mathcal{G}, \text{ and }\forall V\in \mathbb{T}^2, \,\\
   & \restr{u}{\gamma}(V)=\restr{u}{\gamma'}(V) \text{ for any }\gamma,\gamma'\in\mathcal{G}\text{ joint at }V.\end{aligned}\rma.
\]
We endow these two linear spaces with the following inner products
\[
(u,v)_{\mathcal{M}^1(\Gamma_1^{\#})}=\int_{\Gamma_1^{\#}} u v+u_s v_s,\,u,v\in\mathcal{M}^1(\Gamma_1^{\#}),
\]
and 
\[
(u,v)_{\mathcal{M}_{per}^1(\Gamma_1)}=\int_{\Gamma_1} u v+u_s v_s,\,u,v\in\mathcal{M}_{per}^1(\Gamma_1)
\]
respectively, which evidently makes them Hilbert. When the meaning is clear, we simply use \(\mathcal{M}^1\) and \(\mathcal{M}_{per}^1\) to denote the above two spaces.
\end{dfn}

\begin{prop}
There is a canonical inclusion \(\mathcal{M}_{per}^1\hookrightarrow\mathcal{M}^1\) that preserves the inner product. Moreover, \(\mathcal{M}_{per}^1\) is a closed subspace.
\end{prop}
\begin{proof}
We may clearly identify \(\Gamma_1\) with \(\Gamma_1^{\#}\) geometrically in an essential way, that is, there are isometries (obtained from the periodic extension of \(\mathcal{G}\)) from each regular components of \(\Gamma_1^{\#}\) to those of \(\Gamma_1\) with finitely many points deleted (in a translated coordinate system on \(\R^2\), \(\Gamma_1^{\#}\) changes and the deleted points on \(\Gamma_1\) may vary). These isometries make elements in \(\mathcal{M}_{per}^1\) functions on \(\Gamma_1^{\#}\) through composition in a piecewise manner. The conditions in the definition of \(\mathcal{M}_{per}^1\) ensure that such functions must be in \(\mathcal{M}^1\). The rest are routine proofs.
\end{proof}

We present a Poincar\'{e} Inequality on bounded connected patterns. Here ``bounded" refers to the boundedness of the length of all the arcs involved in the pattern and ``connected" refers that given any two points on the pattern there are arcs composing a piecewise \(C^1\) curve in \(\R^2\) connecting them.

\begin{lem}\emph{(\textbf{Poincar\'{e} Inequality on Bounded Connected Patterns})}
Let \(\mathcal{G}\) be a bounded connected pattern, and \(\Gamma\) the union of all arcs. We have
\be
\lb\int_{\Gamma} (u-<u>)^2\rb^{1/2} \lesssim \lb\int_{\Gamma} u_s^2\rb^{1/2},\,\forall u\in \mathcal{M}^1,
\ee
where \(<u>=\int_{\Gamma} u/l\) with \(l\) the total length of arcs in \(\mathcal{G}\).\label{poincare}
\end{lem}

\begin{proof}
For any two points \(x,y\) on the pattern, we have, by connectedness, \(\gamma_1,\cdots,\gamma_k\in\mathcal{G},\,k\in\N_+\) such that \(x\) is contained in \(\gamma_1\) and \(y\) contained in \(\gamma_k\), and each \(\gamma_i\) connects to \(\gamma_{i+1}\) for all \(i=1,\cdots,k-1\) forming a piecewise \(C^1\) curve in \(\R^2\). Denoting by \(V_i,\,i=1,\cdots,k-1\) the intersection points of these curves between \(x\) and \(y\), we have
\[
\begin{split}
  |u(x)-u(y)|&\le \sum_{i=1}^{k-2}|u(V_i)-u({V_{i+1}})| + |u(x)-u(V_1)| + |u(V_{k-1})-u(y)|\\
             &\le \sum_{i=1}^{k-2}\int_{\gamma_{i+1}}|u_s| + \int_{\gamma_1} |u_s| + \int_{\gamma_k} |u_s|\\
             &\le  l^{1/2} \lb\int_{\Gamma} u_s^2\rb^{1/2}.
\end{split}
\]
Now, we have
\[
\int_{\Gamma} (u-<u>)^2=\int_{\Gamma}\lb\frac{\int_{\Gamma}|u(x)-u(y)|d y}{l}\rb^2 d x \le l^2 \int_{\Gamma} u_s^2.
\]
\end{proof}

We wish to consider the semi-inner-product
\[
<u,v>_{\mathcal{M}^1}=\int_{\Gamma_1} u_s v_s,\,u,v\in \mathcal{M}^1,
\]
through which we may define a semi-norm
\[
p(u)=\sqrt{<u,u>_{\mathcal{M}^1}}.
\]
We define
\[
K^1=\lma u\in\mathcal{M}^1;\; p(u)=0\rma,\,K_{per}^1=K^1\cap\mathcal{M}_{per}^1.
\]
According to Lemma (\ref{poincare}), we know that functions in \(K^1\) take constant values on each connected components of \(\Gamma_1\), and because the number of arcs is finite in each cell, \(\text{dim } K_{per}^1\le \text{dim } K^1<\infty\). We further define
\be
M^1=(K^1)^{\perp},\,M_{per}^1=(K^1)^{\perp}\cap \mathcal{M}_{per}^1.
\ee
Also by Poincar\'{e} Inequality, we know that the restrictions of ``\(<,>_{\mathcal{M}^1}\)" to \(M^1\) and \(M_{per}^1\) now are inner products and are equivalent to the original inner products. From now on, we consider \(M^1\) and \(M_{per}^1\) as Hilbert spaces with inner product ``<,>".

\subsubsection{The Estimate}

As in the case \(a\rta0^+\), we still evaluate the term \(-a\int_{\Gamma_1} (\hat{w}_k)_s(x_k)_s\). We may assume that \((\hat{w}_k)_s\) is not constant 0. Notice that \(\hat{w}_k\in \mathcal{M}_{per}^1\) and \(x_k\in \mathcal{M}^1\). Let \(w\) be the projection of \(\hat{w}_k\) into \(M_{per}^1\), and \(z\) the projection of \(x_k\) into \(M^1\). Thus, we can rewrite
\[
\frac{\lw \int_{\Gamma_1} (\hat{w}_k)_s(x_k)_s\rw}{\lb\int_{\Gamma_1}(\hat{w}_k)_s^2\rb^{1/2}}=\frac{\lw \int_{\Gamma_1} w_s z_s\rw}{\lb\int_{\Gamma_1}w_s^2\rb^{1/2}}.
\]
Let \(d_k^{\#}=dist(z,M_{per}^1)\), we obtain by elementary geometry
\[
\frac{\lw \int_{\Gamma_1} w_s z_s\rw}{\lb\int_{\Gamma_1}w_s^2\rb^{1/2}}\le \lb <z,z>-(d_k^{\#})^2  \rb^{1/2},
\]
and then 
\be
{\lw \int_{\Gamma_1} (\hat{w}_k)_s(x_k)_s\rw}\le {\lb\int_{\Gamma_1}(\hat{w}_k)_s^2\rb^{1/2}}\lb \int_{\Gamma_1}(x_k)_s^2-(d_k^{\#})^2   \rb^{1/2}.\label{dkine}
\ee
According to equation (\ref{yyds}), we have
\be
{\int_{\Gamma_1}(\hat{w}_k)_s^2}\le {\lw \int_{\Gamma_1} (\hat{w}_k)_s(x_k)_s\rw}. \label{less}
\ee
Inserting estimate (\ref{less}) into (\ref{dkine}), we obtain
\[
{\lw \int_{\Gamma_1} (\hat{w}_k)_s(x_k)_s\rw}\le \int_{\Gamma_1}(x_k)_s^2 - (d_k^{\#})^2 ,
\]
which shows that
\[
(\Sigma_0)_{kk}\ge 1+ a (d_k^{\#})^2 ,
\]
and hence 
\be
tr(\Sigma_0)\ge 2 + a \lb(d_1^{\#})^2+(d_2^{\#})^2\rb.
\ee
Stopping here will not be enough because we find that in some cases, \((d_1^{\#})^2+(d_2^{\#})^2\) could be zero. Moreover, the translation of coordinate system also possibly change \((d_1^{\#})^2+(d_2^{\#})^2\), and thus, to make a more intrinsic estimate, we take \(d^2\) to be the maximum of \((d_1^{\#})^2+(d_2^{\#})^2\) over all translations of coordinate system.

\subsection{Examples on Patterns with different \(0\le d\le \sqrt{l}\)}

\begin{itemize}
    \item[1.] The trace of the effective diffusion tensor of balanced patterns are exactly \(2+a l\), regardless of the translations of coordinate system;
    \item[2.] There is a pattern having the corresponding \(d=0\). 
    
\begin{figure}[htbp]
        \centering
        \includegraphics[width=9cm]{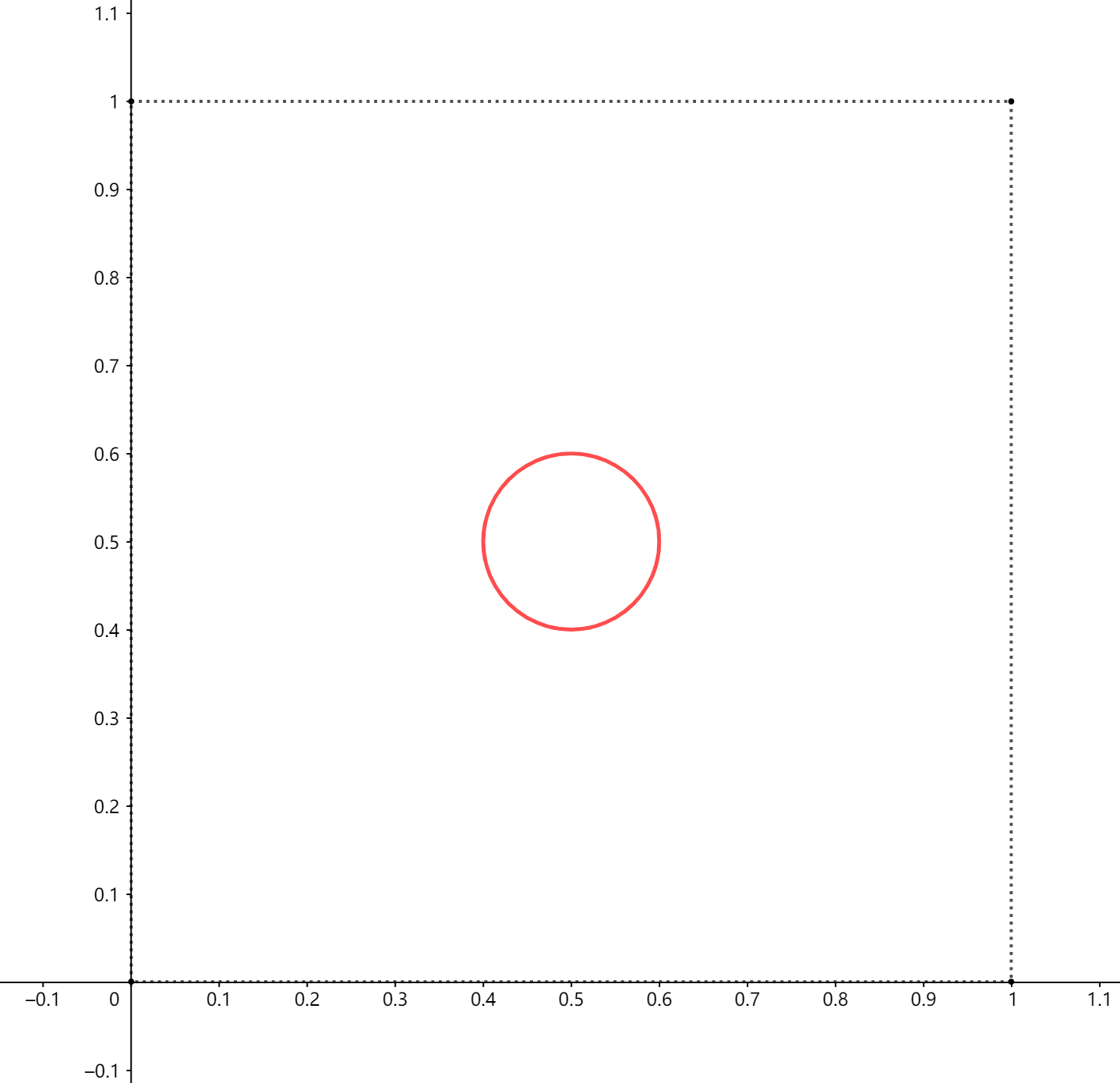}
        \caption{In the center of each cell there is a circle of radius 0.1; this is also considered as \(\Gamma_1^{\#}\) of Type 0}
         \label{circle}
\end{figure}

\begin{figure}[htbp]
     \begin{subfigure}{0.3\textwidth}
         \includegraphics[width=1.4\textwidth]{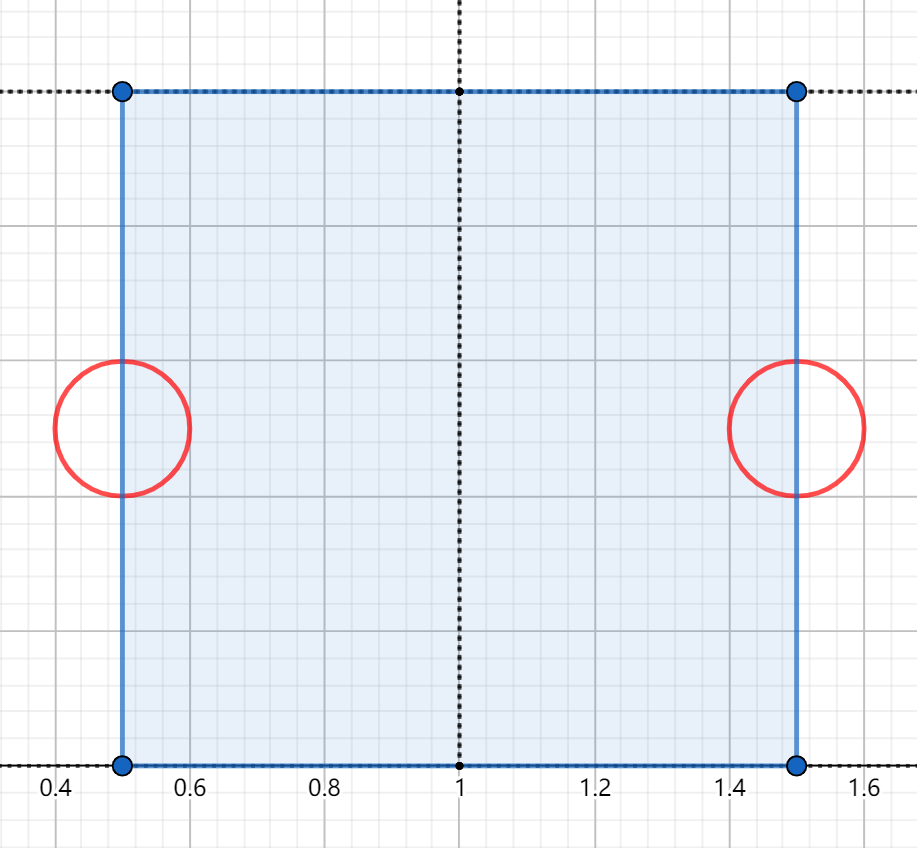}
         \caption{\(\Gamma_1^{\#}\) of Type 1}
         \label{1}
     \end{subfigure}
     \hspace{1.2in}
     \begin{subfigure}{0.3\textwidth}
         \includegraphics[width=1.3\textwidth]{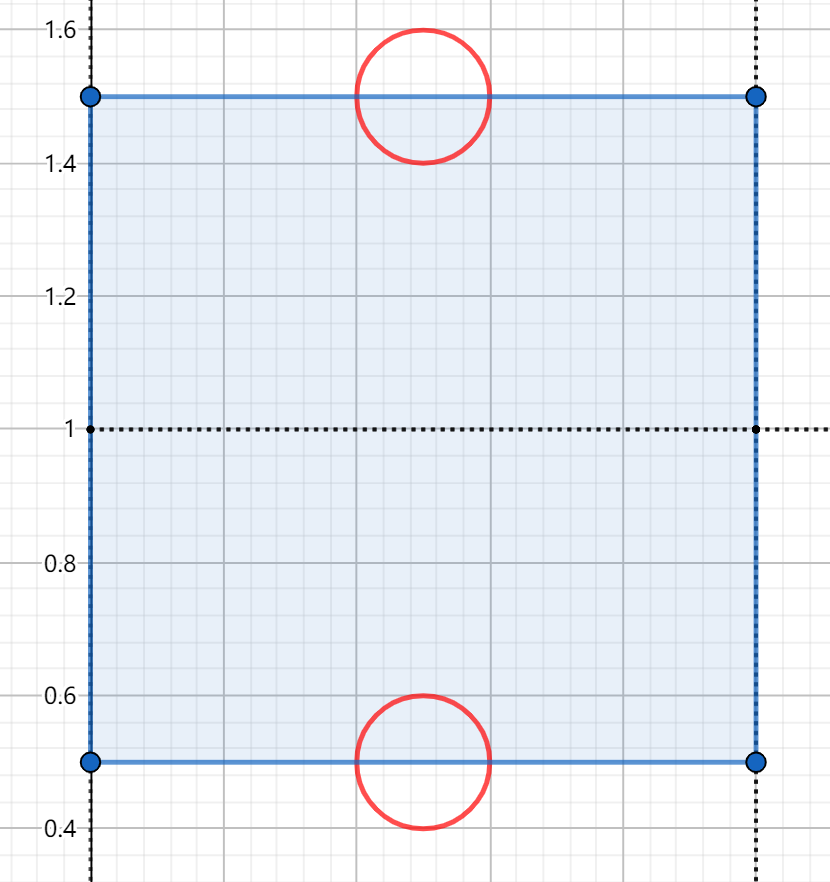}
         \caption{\(\Gamma_1^{\#}\) of Type 2}
         \label{2}
     \end{subfigure}
        \caption{Type 1 and 2 are similar, but have different effects on calculating \(d_k^{\#}\) for a given \(k=1,2\)}
        \label{12}
\end{figure}   
          
\begin{figure}[htbp]
     \begin{subfigure}{0.3\textwidth}
         \includegraphics[width=1.2\textwidth]{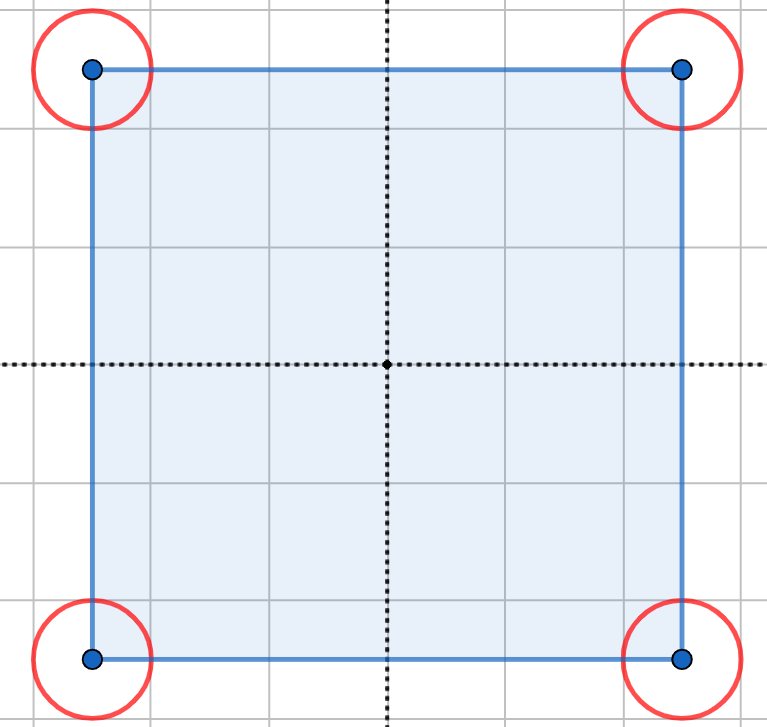}
         \caption{\(\Gamma_1^{\#}\) of Type 3}
         \label{3}
     \end{subfigure}
     \hspace{1.2in}
     \begin{subfigure}{0.3\textwidth}
         \includegraphics[width=1.2\textwidth]{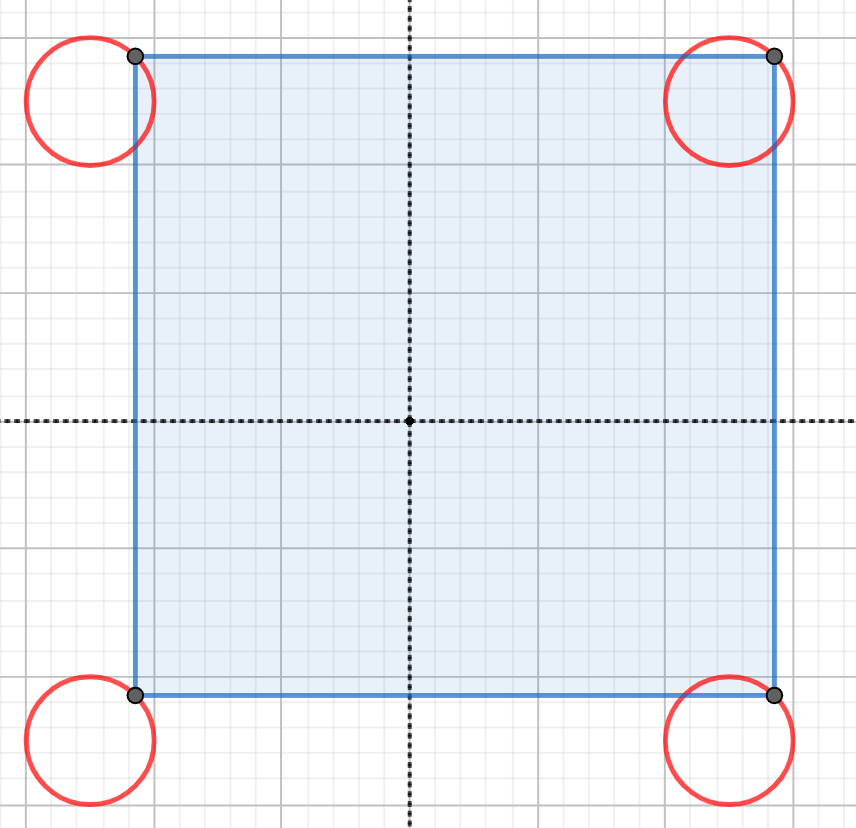}
         \caption{\(\Gamma_1^{\#}\) of Type 4}
         \label{4}
     \end{subfigure}
     \hspace{1.2in}
     \centering
     \begin{subfigure}{0.3\textwidth}
         \includegraphics[width=1.4\textwidth]{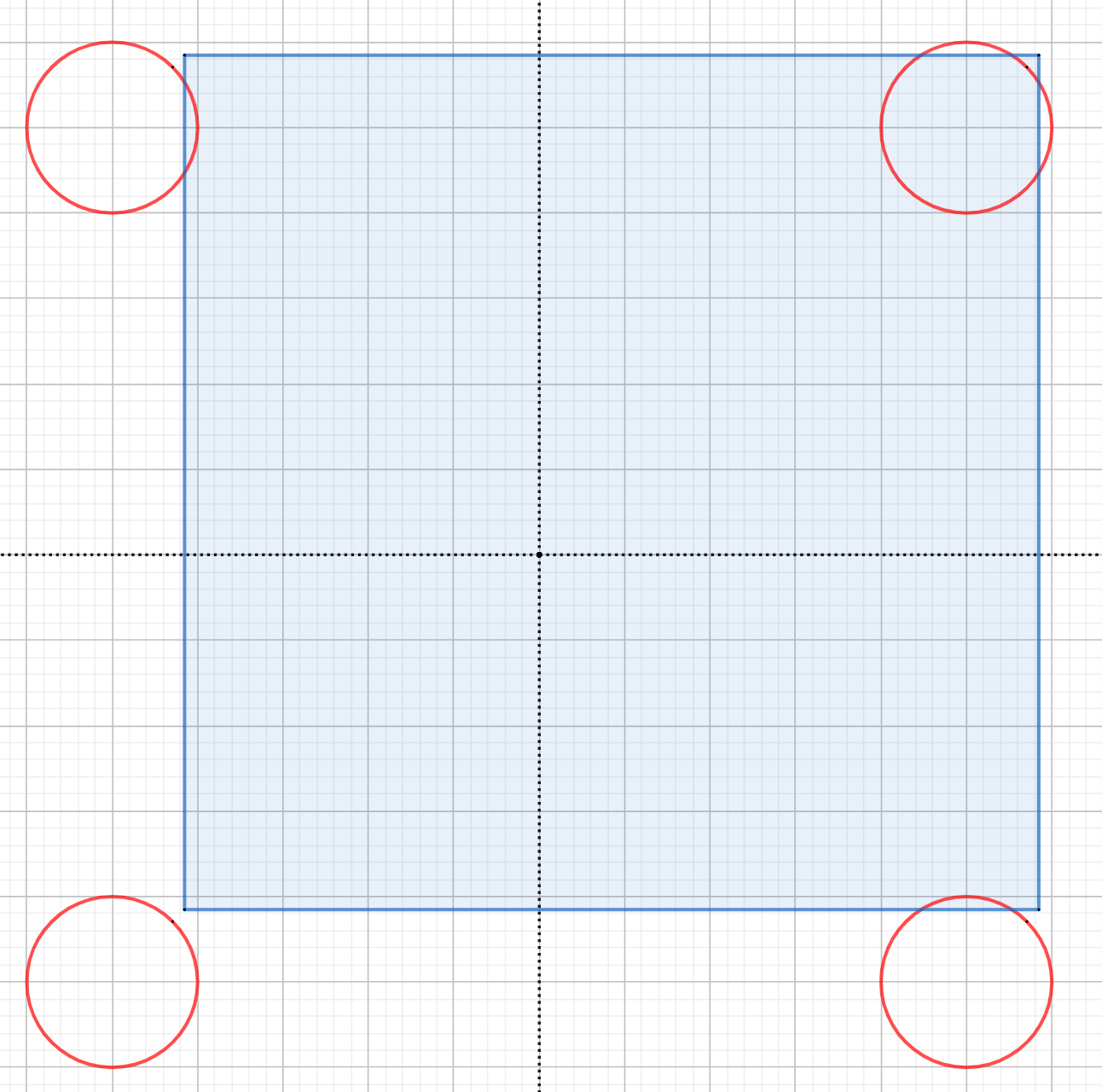}
         \caption{\(\Gamma_1^{\#}\) of Type 5}
         \label{5}
     \end{subfigure}
        \caption{Type 3 and 5 have four connected components, while Type 4 has three.}
        \label{34}
\end{figure}  

\clearpage

    We start computing \(d\) for the above six types of \(\Gamma_1^{\#}\). It is clear that in type 0, \(\mathcal{M}^1\) coincides with \(\mathcal{M}_{per}^1\), and so the distance \(d_k^{\#}=0\) for all \(k=1,2\). In type 1 and 2, the projections of \(x_k\)'s to \(M^1\) are contained in \(M_{per}^1\), although, at this moment, \(x_k\)'s are not in \(\mathcal{M}_{per}^1\), which again forces \(d_k^{\#}=0\). Similar proof also holds for type 3, 4 and 5. It is not hard to see that \(\Gamma_1^{\#}\) under any translation of coordinate system of \(\R^2\) can be classified topologically into one of the six types listed above, and \(d_k^{\#}\)'s remain constant in each class. In conclusion, \(d\) should be 0;
    
    \item[3.] There is a pattern having the corresponding \(0<d<\sqrt{l}\):
    
    \begin{figure}[htbp]
        \centering
        \includegraphics[width=9cm]{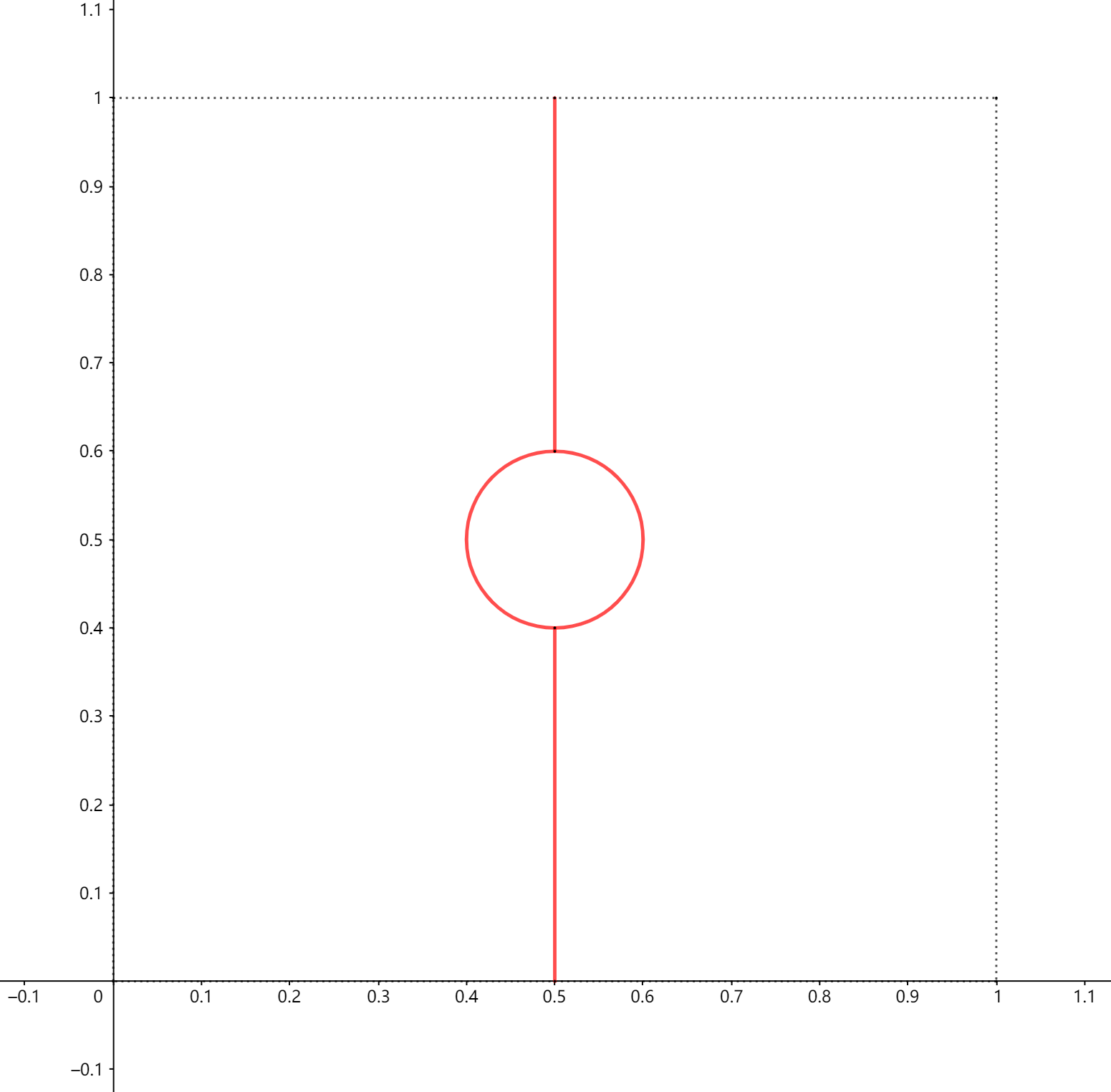}
        \caption{In the center of each cell there is a circle of radius 0.1 and there is one linear segment crossing the boundary of the unit cell and connecting the north and south pole of the circle}
         \label{circleline}
\end{figure}
    
    In this case, \(K^1\) is exactly one dimensional and consists of constant functions only. Observing that \(x_2=y\pm\) some constant \(C\) can never be in \(M_{per}^1\), we see \(d_2^{\#}>0\), and hence \(d>0\). Moreover, the involved circle makes the pattern non-balanced, which ensures that \(d<\sqrt{l}\).
   
\end{itemize}
 
    \begin{thm}
    From example 3., we see that if (in some translated coordinate system) \(\Gamma_1^{\#}\) is connected, and there is an arc of the pattern crossing \(\partial \msquare\), then we have \(0< d \le \sqrt{l}\) such that
    \[
    2+d^2 a \le tr(\Sigma_0) \le 2+ l a.
    \]
    \end{thm}

\clearpage

\section*{Acknowledgements}

My deepest gratitude goes first and foremost to Prof. Xuefeng Wang, my supervisor, for his constant encouragement and guidance. Without his conscientious help, this article could not have come into being. I also thank Prof. Zhen Zhang for many useful discussions and comments on this article.

\bibliographystyle{plain}
\bibliography{ref}

\end{document}